\documentclass[twoside,reqno,10pt]{amsart}
\usepackage{amsmath}    
\usepackage{amssymb}    
\usepackage{amsthm}     
\usepackage{amsfonts}   
\usepackage{mathtools}  
\usepackage{bbm}		
\usepackage{dsfont}		
\usepackage[mathscr]{euscript}

\usepackage{graphicx}   
\usepackage{tikz-cd}    

\usepackage{geometry}   
\usepackage{setspace}   
\usepackage{enumerate}  

\usepackage{cite}

\usepackage{hyperref}   
\usepackage{comment}	

\hypersetup{
	colorlinks=true,
	linkcolor=blue,
	citecolor=blue,
	urlcolor=blue
}

\theoremstyle{definition}
\newtheorem{definition}{Definition}[section]
\newtheorem{theorem}[definition]{Theorem}
\newtheorem{lemma}[definition]{Lemma}
\newtheorem{corollary}[definition]{Corollary}
\newtheorem{remark}[definition]{Remark}

\newtheorem{example}[definition]{Example}
\newtheorem{proposition}[definition]{Proposition}

\title{Universal Coacting Hopf algebra of a Finite dimensional Lie-Yamaguti algebra}

\author{Saikat Goswami}
\address{TCG Centres for Research and Education in Science and Technology, Institute for Advancing Intelligence, Salt Lake, Kolkata-700091, West Bengal, India.}
\address{RKMVERI, Belur, Howrah-711202, West Bengal, India.}
\email{saikatgoswami.math@gmail.com}

\author{Satyendra Kumar Mishra}	
\address{Indian Institute of Technology (BHU), Department of Mathematical Sciences,  Varanasi-221005, India.}
\email{satyamsr10@gmail.com}

\author{Goutam Mukherjee}
\address{TCG Centres for Research and Education in Science and Technology, Institute for Advancing Intelligence, Salt Lake, Kolkata-700091, West Bengal, India.}
\address{Academy of Scientific and Innovative Research (AcSIR), Ghaziabad- 201002, India.}
\email{goutam.mukherjee@tcgcrest.org, gmukherjee.isi@gmail.com}

\subjclass[2020]{16D90, 16T05, 16T10, 17A30, 17A36, 17A60.\\ 
Corresponding Author (Author 2), email: satyamsr10@gmail.com.}

\begin{document}

\maketitle

\vspace{-0.32in}
\begin{abstract}
	M. E. Sweedler first constructed a universal Hopf algebra of an algebra. It is known that the dual notions to the existing ones play a dominant role in Hopf algebra theory. Yu. I. Manin and D. Tambara introduced the dual notion of Sweedler's construction in separate works. In this paper, we construct a universal algebra for a finite-dimensional Lie-Yamaguti algebra. We demonstrate that this universal algebra possesses a bialgebra structure, leading to a universal coacting Hopf algebra for a finite-dimensional Lie-Yamaguti algebra. Additionally, we develop a representation-theoretic version of our results. As an application, we characterize the automorphism group and classify all abelian group gradings of a finite-dimensional Lie-Yamaguti algebra.  
\end{abstract}

\tableofcontents

\vspace{-0.4in}
\section{\large Introduction}	
	This article's central objects of study are Lie-Yamaguti algebras, which generalize Lie triple systems and Lie algebras simultaneously. N. Jacobson \cite{Jacobson} introduced the notion of a Lie triple system as an algebraic object closely associated with quantum mechanics. Lie-Yamaguti algebras arose naturally from K. Nomizu's \cite{nomizu54} work on invariant affine connections on homogeneous spaces. K. Yamaguti \cite{yamaguti58} introduced the notion of Lie-Yamaguti algebras with the name \textit{general Lie triple system} in order to characterize the torsion and curvature tensors of K. Nomizu's \cite{nomizu54} canonical connection. Later, M. Kikkawa \cite{kikk-geo} renamed general Lie triple systems as \textit{Lie triple algebra}. M. Kikkawa in \cite{kikkawa72} showed that Lie-Yamaguti algebras are deeply connected to differential geometry by proving that the category of simply connected, complete, and locally reductive spaces with base points is equivalent to the category of real finite-dimensional Lie triple algebras.  
	
	\medskip
	M. K. Kinyon and A. Weinstein \cite{kinyon-weinstein} observed that Lie triple algebras (they coined the name \textit{Lie-Yamaguti algebras} for Lie triple algebras) can be constructed from Leibniz algebras. A. Sagle \cite{sagle-sim-anti} constructed several examples of Lie-Yamaguti algebras from reductive homogeneous spaces. K. Yamaguti \cite{yama-malcev} showed that Lie-Yamaguti algebras are related to many other non-associative algebras; for example, every Malcev algebra $(L,\langle \cdot, \cdot \rangle)$ gives a Lie-Yamaguti algebra structure on $L$ with the binary operation $[\cdot,\cdot]$ and the ternary operation $\{\cdot,\cdot,\cdot\}$ being 
	\begin{equation}\label{Malcev to LY}
		[x,y]= \langle x,y \rangle \quad \text{and} \quad \{x,y,z\} = \langle x,\langle y,z\rangle\rangle - \langle y,\langle x,z \rangle\rangle + \langle\langle x,y\rangle, z \rangle, \quad \forall x,y,z\in L.
	\end{equation}
	We refer to \cite{abdelwahab22, benito, benito-irreducible,our-paper,ripan-asia, ripan-JOA, kikk-kill,kikk-solv, zhang} for recent advancements in the area of Lie-Yamaguti algebras.  
	
	\medskip
	Let $A$ and $B$ be two unital associative algebras over a field $\mathbb{K}$ of characteristic zero, with $\operatorname{Hom}(A,B)$ denoting the vector space of all linear maps from $A$ to $B$. Let $C$ be a coalgebra over $\mathbb{K}$ with comultiplication $\Delta$ and counit $\epsilon$. A linear map $\phi: C\rightarrow \operatorname{Hom}(A,B)$ is said to be a {\it coalgebra measuring} if 
	\begin{equation*}
		\phi(c)(ab)=\sum \phi(c_1)(a)\cdot \phi(c_2)(b) \qquad \forall ~~ a,b\in A ~~ \mbox{and} ~~ \Delta(c)=\sum c_1\otimes c_2,
	\end{equation*}
	\begin{equation*}
	\phi(c)(1_A)=\epsilon(c)1_B \quad \forall~ c\in C.
	\end{equation*}	
	If such a coalgebra measuring exists, we say that the pair $(C,\phi)$ measures $A$ to $B$ (see \cite[Sec. 7.0, Pg. 139]{sweedler69} for more details). Coalgebra measuring can be realized as a generalization of algebra morphism from $A$ to $B$. Note that for any group-like element $c\in C$, the map $\phi(c):A\rightarrow B$ becomes an algebra morphism. For two algebras $A$ and $B$, M. E. Sweedler introduced the universal coalgebra measuring $M(A,B)$. More precisely, M. E. Sweedler \cite[Proposition 7.0.4]{sweedler69}  proved that for any two algebras $A$ and $B$, there exists a coalgebra $M(A,B)$ and a map $\Theta:  M(A,B) \rightarrow \operatorname{Hom}(A , B)$ that measures $A$ to $B$  with the following universal property: for any coalgebra $C$ and a map $\psi : C\rightarrow \operatorname{Hom}(A,B)$, if  $(C, \psi)$ measures $A$ to $B$, then there is a unique coalgebra map $F:C \to M(A,B)$ such that the following diagram
	\begin{equation*}
		\begin{tikzcd}
			B & {M(A,B) \otimes A} \\
			& {C \otimes A}
			\arrow["\Theta"', from=1-2, to=1-1]
			\arrow["\psi", from=2-2, to=1-1]
			\arrow["{F \otimes id_A}"', from=2-2, to=1-2]
		\end{tikzcd}
	\end{equation*}
	commutes. In the above diagram, we identified a map $\phi: C\rightarrow \operatorname{Hom}(A,B)$ as a map $\phi: C\otimes A\rightarrow B$ (since the $\operatorname{Hom}$ functor and the tensor product functor are adjoint to each other). The universal object $M(A,B)$ enriches the category of algebras over the category of coalgebras. Subsequently, this enrichment leads to universal measuring bialgebra \cite[Chapter 7]{sweedler69}. Sweedler's construction naturally gives a universal (acting) Hopf algebra of an algebra.  

	\medskip
	The dual notion to an existing notion plays an equally important role, we refer the reader to \cite{Bergh} to see their importance in non-commutative algebraic geometry. The dual version of the Sweedler's work was considered by Tambara in \cite[Theorem 1.1]{tambara90} and in the graded algebra case by Manin in \cite{manin88}. Tambara proved that if $A$ is a finite-dimensional algebra then the tensor functor $A \otimes - : \operatorname{Alg}_{\mathbb{K}} \to \operatorname{Alg}_{\mathbb{K}}$ has a left adjoint functor, denoted by $a(A,-)$. It turns out that the algebra $a(A,A)$ has a bialgebra structure. The Hopf envelope (Hopf algebra generated by a bialgebra) of $a(A,A)$ plays the role of a symmetry group in non-commutative geometry. G. Militaru discussed applications of universal constructions in the case of associative algebras in \cite{mil22}. Recently,  A. L. Agore and G. Militaru have affirmatively investigated Tambara's results for the case of Lie/Leibniz algebras \cite{agore20}. In subsequent work, A. L. Agore also studied a functor between the category of Lie algebra modules and the category of algebra modules in \cite{agore23}. Her study leads to a representation-theoretic counterpart of Manin-Tambara's universal coacting objects \cite{manin88, tambara90}. 

	\medskip
	In this paper, our first aim is to show that we can construct a universal coacting Hopf algebra for a finite-dimensional Lie-Yamaguti algebra. This construction gives a left adjoint functor to the current Lie-Yamaguti algebra functor (see Example \ref{DEF_current_LYA} for the definition). As an application, we characterize automorphisms of a finite-dimensional Lie-Yamaguti algebra in terms of its universal coacting Hopf algebra. We also classify abelian group gradings of a Lie-Yamaguti algebra using our construction. The article's second aim is to discuss the representation theoretic counterpart of Manin-Tambara's universal constructions in the finite-dimensional Lie-Yamaguti algebras category.

	\medskip
	The section-wise description of the paper is as follows: In Section \ref{sec-2}, we recall definitions and examples of Lie-Yamaguti algebras. We also recall some basic definitions and results on bialgebras and Hopf algebras that we use throughout the article. Section \ref{sec-3} is devoted to the construction of universal algebra $\mathcal{A}(\mathfrak{L},\mathfrak{K})$ of a finite-dimensional Lie-Yamaguti algebra $\mathfrak{L}$ and an arbitrary Lie-Yamaguti algebra $\mathfrak{K}$. If $\mathfrak{L}=\mathfrak{K}$, then we show that this universal algebra has an underlying bialgebra structure $\mathcal{A}(\mathfrak{L})$ coacting on the Lie-Yamaguti algebra $\mathfrak{L}$. In the sequel, we deduce that there exists a universal coacting Hopf algebra $\mathcal{H}(\mathfrak{L})$ for the Lie-Yamaguti algebra $\mathfrak{L}$. In Section \ref{sec-4}, we first show that the tensor product of an $\mathfrak{L}$-module and an $\mathcal{A}(\mathfrak{L},\mathfrak{K})$-module has a $\mathfrak{K}$-module structure. Then in Theorem \ref{THM_univ_Amodule_U_and_V} we construct a universal $\mathcal{A}(\mathfrak{L},\mathfrak{K})$-module of an $\mathfrak{L}$-module $U$ and $\mathfrak{K}$-module $V$ when $U$ is finite dimensional. Finally, Theorem \ref{THM_univ_module_functor} shows that this construction is functorial. Section \ref{sec-5} presents our results' applications. In particular, we characterize the automorphism group of a finite-dimensional Lie-Yamaguti algebra and classify the abelian group gradings on a finite-dimensional Lie-Yamaguti algebra.

\medskip
\section{\large Preliminaries} \label{sec-2}
	Throughout the paper, all vector spaces, (bi)linear maps, algebras, and bialgebras are over a fixed $\mathbb{K}$, an arbitrary field of characteristic $0$. By a ``basis", we always mean vector space basis over $\mathbb{K}$, and by a (co)algebra, we always mean (co)unital, (co)associative algebra. We denote by $\operatorname{Alg}_{\mathbb{K}}$, the category of all algebras over $\mathbb{K}$ and by $\operatorname{ComAlg}_{\mathbb{K}}$, the category of all commutative algebras over $\mathbb{K}$. This section recalls all the necessary definitions and notions required in this paper. 
	
\medskip
\subsection*{\large Lie-Yamaguti algebras}
	\begin{definition}
		A {\bf Lie-Yamaguti Algebra} over $\mathbb{K}$ is a triple $(\mathfrak{L},[\cdot,\cdot],\{\cdot,\cdot,\cdot\})$, where $\mathfrak{L}$ is a vector space over $\mathbb{K}$ equipped with a bilinear product $[\cdot,\cdot]:\mathfrak{L} \times \mathfrak{L} \to \mathfrak{L}$ and a trilinear product $\{\cdot,\cdot,\cdot\}:\mathfrak{L} \times \mathfrak{L} \times \mathfrak{L} \to \mathfrak{L}$ such that the following conditions hold
		\begin{equation}
			[a,b] = -[b,a], \tag{LY1} \label{LY1}
		\end{equation}
		\begin{equation}
			\{a,b,c\} = -\{b,a,c\}, \tag{LY2} \label{LY2}
		\end{equation}
		\begin{equation}
			[[a,b],c]+\{a,b,c\}+[[b,c],a]+\{b,c,a\}+[[c,a],b]+\{c,a,b\}=0,
			\tag{LY3} \label{LY3}
		\end{equation} 
		\begin{equation}
			\{[a,b],c,d\}+\{[b,c],a,d\}+\{[c,a],b,d\} = 0, 
			\tag{LY4} \label{LY4}
		\end{equation}
		\begin{equation}
			\{a,b,[c,d]\} = [\{a,b,c\},d]+[c,\{a,b,d\}] 
			\tag{LY5} \label{LY5},
		\end{equation}
		\begin{equation}
			\{a,b,\{c,d,e\}\} = \{\{a,b,c\},d,e\}+\{c,\{a,b,d\},e\}+\{c,d,\{a,b,e\}\} 
			\tag{LY6} \label{LY6}, 
		\end{equation}
		for all $a,b,c,d,e \in \mathfrak{L}$.
	\end{definition}
	
	\noindent
	For any two Lie-Yamaguti algebra $\mathfrak{L}$ and $\mathfrak{L}'$, a linear map  $\varphi:\mathfrak{L} \to \mathfrak{L}'$ is a \textbf{Lie-Yamaguti algebra morphism} if 
	\begin{align*}
		\varphi([a,b]) = [\varphi(a),\varphi(b)] \quad\mbox{and }\quad \varphi(\{a,b,c\}) = \{\varphi(a),\varphi(b),\varphi(c)\}, \quad \forall~ a,b,c \in \mathfrak{L}.
	\end{align*}
	It is said to be an isomorphism if $\varphi$ is invertible. We denote the category of all Lie-Yamaguti algebras over $\mathbb{K}$ by $\operatorname{LYA}_\mathbb{K}$ and the category of all Lie/Leibniz algebras over $\mathbb{K}$ by $\operatorname{Lie}_{\mathbb{K}}$ and $\operatorname{Leib}_{\mathbb{K}}$. 
	
	\begin{example}
		Let $(\mathfrak{g},[\cdot,\cdot])$ be a Lie algebra. Define a binary and ternary bracket on $\mathfrak{g}$ by
		\[
		[a,b]:=[a,b], \quad \{a,b,c\} := [[a,b],c],\quad \forall a,b,c\in \mathfrak{g}. 
		\]
		Then $(\mathfrak{g},[\cdot,\cdot],\{\cdot,\cdot,\cdot\})$ is a Lie-Yamaguti algebra. 
	\end{example}
	
	\begin{example}
		Let $(\mathfrak{r},\langle \cdot,\cdot \rangle)$ be a reductive Lie algebra with decomposition $\mathfrak{r}=\mathfrak{g} \oplus \mathfrak{h}$, then $\langle\mathfrak{g},\mathfrak{h}\rangle \subseteq \mathfrak{h}$ and $\langle \mathfrak{g},\mathfrak{g}\rangle \subseteq \mathfrak{g}$. Let $\pi_{\mathfrak{h}}:\mathfrak{r} \to \mathfrak{h}$ and $\pi_{\mathfrak{g}}:\mathfrak{r} \to \mathfrak{g}$ be the projection maps onto $\mathfrak{h}$ and $\mathfrak{g}$ respectively. Then $\mathfrak{h}$ is a Lie-Yamaguti algebra with the brackets
		\[
		[a,b] := \pi_{\mathfrak{h}}\langle a,b \rangle, 
		\quad \{a,b,c\} := \langle\pi_{\mathfrak{g}}\langle a,b\rangle,c\rangle,\quad\forall ~a,b,c\in \mathfrak{h}.
		\]
	\end{example}
	
	\begin{example} 
		Let $(\mathfrak{g},\cdot)$ be a Leibniz algebra. Then $\mathfrak{g}$ is a Lie-Yamaguti algebra with the following binary and ternary brackets
		\[[a,b]  =  a\cdot b - b \cdot a,  \quad
		\{a,b,c\}  =  -(a \cdot b) \cdot c, \quad\forall ~ a,b,c\in \mathfrak{g}.\]
	\end{example}
	
	Let $\mathfrak{L}$ be a Lie-Yamaguti algebra and $A$ be a commutative algebra. We then construct a Lie-Yamaguti algebra structure on $\mathfrak{L} \otimes A$ (see Example \ref{DEF_current_LYA}). We call it {\bf current Lie-Yamaguti algebra}, following the naming conventions of Lie and Leibniz algebras. 
	
	\begin{example} \label{DEF_current_LYA}
		Let $(\mathfrak{L}, [\cdot,\cdot], \{\cdot,\cdot,\cdot\})$ be a Lie-Yamaguti algebra and $A$ be a commutative algebra. Then there is a Lie-Yamaguti algebra structure on $\mathfrak{L} \otimes A$ given by 
		\begin{equation*}
			\left[ x \otimes a, y \otimes b \right]_{\mathfrak{L} \otimes A} : = \left[ x,y \right] \otimes ab \quad \text{and} \quad \{x \otimes a, y \otimes b, z \otimes c\}_{\mathfrak{L} \otimes A} := \{x,y,z\} \otimes abc, \quad \forall~ x,y,z \in \mathfrak{L},~ a,b,c \in A.
		\end{equation*} 
		
		\noindent
		Further, let $A$ and $B$ be two commutative algebras. If $f: A \to B$ is an algebra morphism, then $id_{\mathfrak{L}} \otimes f : \mathfrak{L} \otimes A \to \mathfrak{L} \otimes B$ is a Lie-Yamaguti algebra morphism. Hence for a fixed Lie-Yamaguti algebra $\mathfrak{L}$ assigning $A \mapsto \mathfrak{L} \otimes A$ defines a covariant functor $\mathfrak{L} \otimes - : \operatorname{ComAlg}_{\mathbb{K}} \to \operatorname{LYA}_\mathbb{K}$. We call this functor the {\bf current Lie-Yamaguti algebra functor}.
	\end{example} 
	
	\begin{definition}
		Let $\mathfrak{L}$ be a Lie-Yamaguti algebra. A \textbf{Lie-Yamaguti algebra module} of $\mathfrak{L}$ is a vector space $V$ along with a linear map $\rho:\mathfrak{L} \to \operatorname{End}(V)$ and two bilinear maps $D,\theta:\mathfrak{L} \times \mathfrak{L} \to \operatorname{End}(V)$ such that 
		\begin{align}
			&D(a,b) + \theta(a,b) - \theta(b,a) = 	[\rho(a),\rho(b)] - \rho([a,b]), \label{R1} \tag{R1} \\
			&\theta(a,[b,c]) - \rho(b) \theta(a,c) + \rho(c) 	\theta(a,b) = 0, \label{R2} \tag{R2} \\
			&\theta([a,b],c) - \theta(a,c) \rho(b) + \theta(b,c) 	\rho(a) = 0,
			\label{R3} \tag{R3} \\
			&\theta(c,d) \theta(a,b) - \theta(b,d) \theta(a,c) - 	\theta(a,\{b,c,d\}) + D(b,c) \theta(a,d) = 0, 
			\label{R4} \tag{R4} \\
			&[D(a,b),\rho(c)] = \rho(\{a,b,c\}),
			\label{R5} \tag{R5} \\
			&[D(a,b),\theta(c,d)] = \theta(\{a,b,c\},d) + 	\theta(c,\{a,b,d\}), \label{R6} \tag{R6}
		\end{align}
		for all $a,b,c,d \in \mathfrak{L}$.
	\end{definition}
	\noindent
	Using the relations (\ref{R1}), (\ref{R2}), (\ref{R3}), and (\ref{R5}) stated above, we get an additional identity 
	\begin{equation} \tag{R7}
		D([a,b],c) + D([b,c],a) + D([c,a],b) = 0.
	\end{equation}
	
	\smallskip 
	\noindent
	Taking $V=\mathfrak{L}$, there is a natural module structure of $\mathfrak{L}$ on itself, where the representation maps $\rho, D, \theta$ are given by 
	\begin{equation*}
		\rho(a)(b) := [a,b], \quad D(a,b)(c) := \{a,b,c\}, \quad\mbox{and }\quad \theta(a,b)(c) :=  \{c,a,b\}, \quad \forall~ a,b,c \in \mathfrak{L}. 
	\end{equation*}
	
\medskip
\subsection*{\large Bialgebras and Hopf algebras}~
	
	\medskip\noindent
	Before going to bialgebras and Hopf algebras, we first recall the definition of coalgebra and coalgebra modules. 

	\begin{definition} 
		A {\bf coalgebra} over $\mathbb{K}$ is a vector space $C$ over $\mathbb{K}$ along with linear maps, $\Delta: C \to C \otimes C$ (comultiplication map) and $\varepsilon : C \to \mathbb{K}$ (counital map) such that the following conditions hold
		\begin{enumerate}
			\item (Coassociativity condition) $(id \otimes \Delta) \circ \Delta ~=~ (\Delta \otimes id) \circ \Delta$,
			\item (Counit) $(id \otimes \varepsilon) \circ \Delta ~=~ id ~=~ (\varepsilon \otimes id) \circ \Delta$.
		\end{enumerate}
		
	 	\noindent
	\end{definition}
	
	\begin{definition} 
		Let $(C,\Delta, \varepsilon)$ and $(D,\Delta',\varepsilon')$ be two coalgebras over $\mathbb{K}$. A linear map $f:C \to D$ is called a {\bf coalgebra morphism} if $f$ satisfies the conditions: $\Delta' \circ f = (f \otimes f) \circ \Delta$ and $\varepsilon' \circ f = \varepsilon$.	
	\end{definition}
	
We denote the category of coalgebras over $\mathbb{K}$ by $\operatorname{Cog}_{\mathbb{K}}$.	
	\begin{definition}
		Let $(C, \Delta, \varepsilon)$ be a coalgebra over $\mathbb{K}$. A {\bf right comodule over $C$} is a vector space $M$ together with a linear map $f: M \to M \otimes C$ such that
		\begin{equation*}
			(id \otimes \Delta) \circ f = (f \otimes id) \circ f \quad \text{and}
			\quad (id \otimes \varepsilon) \circ f = id
		\end{equation*}
	\end{definition}
	
	\medskip\noindent
	{\bf Duality between algebras and coalgebras:}
	The notion of algebras and coalgebras are dual in the sense that their axioms are dual to each other (obtained by reversing the arrows). The dual of a coalgebra is an algebra, while in general, the dual of an algebra may not be a coalgebra. Assuming the algebra to be finite-dimensional, its dual becomes a coalgebra. We now briefly sketch the construction of the dual of a coalgebra and discuss the finite dual functor. For a more detailed description, see \cite[Sec. 6.0, Pg. 109]{sweedler69}.  
	
	\medskip
	Let $(C,\Delta,\varepsilon)$ be a coalgebra over $\mathbb{K}$. Then the dual of $C$, that is $C^* = \operatorname{Hom}(C,\mathbb{K})$, all linear maps from $C$ to $\mathbb{K}$, becomes an algebra with the convolution product: for any $f, g \in \operatorname{Hom}(C, \mathbb{K})$ and for any $c \in C$ with $\Delta(c) = \sum_{(c)} c_{(1)} \otimes c_{(2)} \in C \otimes C$ then 
	\[f \star g(c) := \sum_{(c)} f(c_{(1)}) g(c_{(2)})\]
	The product $\star$ makes $\operatorname{Hom}(C,\mathbb{K})$ an algebra, which is called the {\bf algebraic dual of the coalgebra $C$}. The unit element in $\operatorname{Hom}(C,\mathbb{K})$ is the map $\mathbbm{1}: C \to \mathbb{K}$ given by $\mathbbm{1}(c) = 1_{\mathbb{K}}$ for all $c \in C$. This construction extends to a contravariant functor $()^*: \operatorname{Cog}_{\mathbb{K}} \to \operatorname{Alg}_{\mathbb{K}}$, where for any coalgebra morphism $f:C \to D$ the algebra morphism $f^*:D^* \to C^*$ is the transpose map.
	
	\medskip
	Now, for an algebra $(A,m)$, its vector space dual $A^*$ is not necessarily a coalgebra; the natural candidate for the comultiplication $\Delta$  is the transpose map $m^*:A^* \to (A \otimes A)^*$ of the multiplication map $m: A \otimes A \to A$. Nevertheless, note that $(A \otimes A)^*$ is isomorphic to $A^* \otimes A^*$ as a vector space if and only if $A$ is finite-dimensional. Hence, we need the notion of finite dual. The {\bf 
		finite dual functor} $()^{\circ} : \operatorname{Alg}_{\mathbb{K}} \to \operatorname{Cog}_{\mathbb{K}}$ is the left adjoint functor  to $()^*: \operatorname{Cog}_{\mathbb{K}} \to \operatorname{Alg}_{\mathbb{K}}$, defined by 
		\[A^{\circ} := \{f \in A^* : \operatorname{Ker}(f)~ \text{contains some ideal} ~ I ~\text{of}~ A ~\text{with}~ \operatorname{dim}_{\mathbb{K}}(A/I) < \infty\} \quad \text{for any}~ A \in \operatorname{Alg}_{\mathbb{K}}.\]  
	
	\medskip
	Next, we recall the definition of a bialgebra and a Hopf algebra over $\mathbb{K}$.
	\begin{definition}
		A {\bf bialgebra} over a field $\mathbb{K}$ is a vector space $B$ over $\mathbb{K}$ equipped with linear maps (multiplication) $\nabla : B \otimes B \to B$, (unit) $\eta: \mathbb{K} \to B$, (comultiplication) $\Delta: B \to B \otimes B$, and (counit) $\varepsilon: B \to \mathbb{K}$ such that $(B,\nabla,\eta)$ is an algebra and $(B,\Delta,\varepsilon)$ is a coalgebra along with the following compatibility conditions
		
		\begin{enumerate}
			\item a compatibility relation between multiplication $\nabla$ and comultiplication $\Delta$ given by the commutative diagram
			\[\begin{tikzcd}
				{B \otimes B} & B & {B \otimes B} \\
				{B \otimes B \otimes B \otimes B} && {B \otimes B \otimes B \otimes B}
				\arrow["\nabla", from=1-1, to=1-2]
				\arrow["\Delta", from=1-2, to=1-3]
				\arrow["{id \otimes \mu \otimes id }", from=2-1, to=2-3]
				\arrow["{\Delta \otimes \Delta}"', 	from=1-1, to=2-1]
				\arrow["{\nabla \otimes \nabla}"', 	from=2-3, to=1-3],
			\end{tikzcd}\]
			where $\mu: B \otimes B \to B$ is the linear map defined by $\mu(x \otimes y) = y \otimes x$ for all $x, y \in B$,
			\smallskip
			\item compatibility between multiplication $\nabla$ and counit $\varepsilon$, and comultiplication $\Delta$ and unit $\eta$ given by the commutative diagrams
			\[\begin{tikzcd}
				{B \otimes B} & B &&& {\mathbb{K} 	\otimes \mathbb{K} \cong \mathbb{K}} \\
				& {\mathbb{K} \otimes \mathbb{K} \cong 	\mathbb{K},} &&& {B \otimes B} & {B,}
				\arrow["\nabla", from=1-1, to=1-2]
				\arrow["\Delta"', from=2-6, to=2-5]
				\arrow["{\varepsilon \otimes 	\varepsilon}"', from=1-1, to=2-2]
				\arrow["\varepsilon", from=1-2, to=2-2]
				\arrow["{\eta \otimes \eta}"', 	from=1-5, to=2-5]
				\arrow["\eta", from=1-5, to=2-6]
			\end{tikzcd}\]
			
			\item and a compatibility relation between unit $\eta$ and counit $\varepsilon$ is given by the commutative diagram
			\[\begin{tikzcd}
				& B \\
				\mathbb{K} && \mathbb{K.}
				\arrow["id", from=2-1, to=2-3]
				\arrow["\eta", from=2-1, to=1-2]
				\arrow["\varepsilon", from=1-2, to=2-3]
			\end{tikzcd}\]
		\end{enumerate}
		These commutative diagrams imply that either ``comultiplication and counit are homomorphisms of algebras" or, equivalently, ``multiplication and unit are homomorphisms of coalgebras." 
	\end{definition}
	
	For any two bialgebra $B$ and $C$ over $\mathbb{K}$ a {\bf bialgebra morphism} $f: B \to C$ is a linear map that is both an algebra and a coalgebra morphism.  
	
	\begin{definition}
		A {\bf Hopf algebra} is a bialgebra $H$ over $\mathbb{K}$ together with a linear map $s: H \to H$ (called the antipode) such that the following diagram
		\[\begin{tikzcd}
			{H \otimes H} && {H \otimes H} \\
			H & \mathbb{K} & H \\
			{H \otimes H} && {H \otimes H}
			\arrow["\varepsilon", from=2-1, to=2-2]
			\arrow["\eta", from=2-2, to=2-3]
			\arrow["\Delta", from=2-1, to=1-1]
			\arrow["{s \otimes id}", from=1-1, to=1-3]
			\arrow["\nabla", from=1-3, to=2-3]
			\arrow["\Delta"', from=2-1, to=3-1]
			\arrow["{id \otimes s}", from=3-1, to=3-3]
			\arrow["\nabla"', from=3-3, to=2-3]
		\end{tikzcd}\]
		commutes. Here, $\Delta$ is the comultiplication map, $\nabla$ is the multiplication map, $\eta$ is the unit, and $\varepsilon$ is the counit of the bialgebra $H$.
	\end{definition} 
	We now recall from \cite{takeuchi71} the construction of a free commutative Hopf algebra from a coalgebra. 
	
	\bigskip\noindent
	{\bf Free Hopf algebra $H(C)$ generated by a coalgebra $C$} \cite[pg. 562]{takeuchi71} 
	
	\smallskip\noindent
	For any vector space $V$, the tensor algebra $T(V)$ has a natural bialgebra structure. Let $C$ be a coalgebra over $\mathbb{K}$. Let $(V_i)_{i \ge 0}$ be a sequence of coalgebras as follows
	\[V_0 = C, \quad V_{i+1} = V_i^{op}.\]
	So the sequence is of the form $(C, C^{op}, C, C^{op},\ldots)$. Let $V:= \sum_{i=0}^{\infty} V_i$ be the direct sum of coalgebras, again a coalgebra. Let $S: V \to V^{op}$ be the coalgebra map $(x_0,x_1,x_2,\ldots) \mapsto (0,x_0,x_1,x_2,\ldots)$. $S$ induces a bialgebra map $T(V) \to T(V)^{op}$. Let $I$ be the two sided ideal of $T(V)$ generated by the elements $\sum x_{(1)} S(x_{(2)}) - \varepsilon(x)1$ and $\sum S(x_{(1)}) x_{(2)} - \varepsilon(x)1$, for $x \in V$. Define $H(C):=T(V)/I$. Note that $S(I) \subset I$. Thus, $H(C)$ is a bialgebra and $S: H(C) \to H(C)^{op}$ is a bialgebra map. It is easy to verify that $S$ is an antipode of $H(C)$, making $H(C)$ a Hopf algebra.  
	
	\begin{definition} \label{DEF_bi-alg_to_hopf-alg}
		$H(C)$ constructed above is said to be \textbf{the free Hopf algebra generated by a coalgebra $C$}. 
	\end{definition}

\medskip
\section{\large The universal coacting Hopf algebra} \label{sec-3}
	In this section, we construct a universal coacting Hopf algebra of a finite-dimensional Lie-Yamaguti algebra following the approach of \cite{tambara90,agore20}. First, let us recall from Example \ref{DEF_current_LYA} that for a Lie-Yamaguti algebra $\mathfrak{L}$, there exists a functor $\mathfrak{L} \otimes -: \operatorname{ComAlg}_{\mathbb{K}} \to \operatorname{LYA}_{\mathbb{K}}$ from the category of commutative algebras to the category of Lie-Yamaguti algebras, called the current Lie-Yamaguti algebra functor. We have the following result if $\mathfrak{L}$ is finite-dimensional. 
	
	\begin{theorem}
		Let $\mathfrak{L}$ be a finite-dimensional Lie-Yamaguti algebra. Then the current Lie-Yamaguti algebra functor $\mathfrak{L} \otimes - : \operatorname{ComAlg}_\mathbb{K} \rightarrow \operatorname{LYA}_\mathbb{K}$ has a left adjoint denoted by $\mathcal{A}(\mathfrak{L}, -) : \operatorname{LYA}_\mathbb{K} \rightarrow \operatorname{ComAlg}_\mathbb{K}$.
	\end{theorem}
	\begin{proof} 
		Let $\{e_1,e_2,\ldots,e_n\}$ be a basis of $\mathfrak{L}$. Let $\{\tau_{ij}^s : i,j,s = 1,2,\ldots,n\} \cup \{\omega_{ijk}^s : i,j,k,s = 1,2,\ldots,n\}$ be the structure constants of $\mathfrak{L}$, i.e., 
		\begin{equation} \label{EQN_LYA_L_struct_const}
			[e_i,e_j] = \sum_{s=1}^n \tau_{ij}^s e_s ~~ \forall~ i,j=1,2,\ldots,n \quad\text{and}\quad \{e_i,e_j,e_k\} = \sum_{s=1}^n \omega_{ijk}^s e_s, ~~ \forall~ i,j,k = 1,2,\ldots,n.
		\end{equation}
		Let $\mathfrak{K}$ be an arbitrary  Lie-Yamaguti algebra; using it, we first construct a commutative algebra $\mathcal{A}(\mathfrak{L},\mathfrak{K})$. Let $\{f_i : i \in I\}$ ($I$ need not be finite) be a basis of $\mathfrak{K}$. We consider the structure constants of $\mathfrak{K}$ as follows: for any $i,j \in I$, let $B_{ij} \subset I$ be a finite subset of $I$ such that
		\begin{equation} \label{EQN_LYA_K_struct_const1}
			[f_i,f_j]_{\mathfrak{K}} = \sum_{u \in B_{ij}} \alpha_{ij}^u f_u,
		\end{equation}  
		and for any $i,j,k \in I$, let $C_{ijk} \subset I$ be a finite subset of $I$ for which we have
		\begin{equation}\label{EQN_LYA_K_struct_const2}
			\{f_i,f_j,f_k\}_{\mathfrak{K}} = \sum_{u \in C_{ijk}} \beta_{ijk}^u f_u.
		\end{equation}
		Consider the polynomial algebra $\mathbb{K}[X_{si} : s=1,2,\ldots,n, i \in I]$ with coefficients in $\mathbb{K}$ and define 
		\[\mathcal{A}(\mathfrak{L},\mathfrak{K}) := \mathbb{K}[X_{si} : s=1,2,\ldots,n, i \in I] / J,\]
		where $J$ is the ideal generated by all the polynomials of the form 
		\begin{equation}\label{EQN_LYA_univ_Poly1}
			P_{(a,i,j)}^{(\mathfrak{L},\mathfrak{K})} := \sum_{u \in B_{ij}} \alpha_{ij}^u X_{au} - \sum_{s,t = 1}^n \tau_{st}^{a} X_{si} X_{tj}, \quad a = 1,\ldots,n ~~\text{and}~~ i,j \in I.
		\end{equation}
		\begin{equation}\label{EQN_LYA_univ_Poly2}
			Q_{(a,i,j,k)}^{(\mathfrak{L},\mathfrak{K})} := \sum_{u \in C_{ijk}} \beta_{ijk}^u X_{au} - \sum_{r,s,t = 1}^n \omega_{rst}^{a} X_{ri} X_{sj} X_{tk}, \quad a = 1,\ldots,n ~~\text{and}~~ i,j,k \in I.
		\end{equation}
		Denoting by $x_{si} := \widehat{X_{si}}$, the class of $X_{si}$ in $\mathcal{A}(\mathfrak{L},\mathfrak{K})$, we have the following expressions
		\begin{equation}\label{EQN_LYA_univ_poly1}
			\sum_{u \in B_{ij}} \alpha_{ij}^u x_{au} ~=~ \sum_{s,t = 1}^n \tau_{st}^{a} x_{si} x_{tj}, \quad a = 1,\ldots,n ~~\text{and}~~ i,j \in I,
		\end{equation}
		\begin{equation}\label{EQN_LYA_univ_poly2}
			\sum_{u \in C_{ijk}} \beta_{ijk}^u x_{au} ~=~ \sum_{r,s,t = 1}^n \omega_{rst}^{a} x_{ri} x_{sj} x_{tk}, \quad a = 1,\ldots,n ~~\text{and}~~ i,j,k \in I.
		\end{equation}
		
		\medskip\noindent
		With the following diagram 
		\[\begin{tikzcd}
			{\operatorname{ComAlg}_\mathbb{K}} && {\operatorname{LYA}_\mathbb{K},}
			\arrow["{\mathfrak{L} \otimes -}", from=1-1, to=1-3]
			\arrow["{\mathcal{A}(\mathfrak{L}, -)}", bend right=-40, from=1-3, to=1-1]
		\end{tikzcd}\]
		
		\medskip\noindent
		we observe that for any $\mathfrak{K} \in \operatorname{LYA}_\mathbb{K}$, we have 
		$\Big( (\mathfrak{L} \otimes -) \circ (\mathcal{A(\mathfrak{L},-)})\Big)(\mathfrak{K}) = \mathfrak{L} \otimes \mathcal{A}(\mathfrak{L},\mathfrak{K})$.

		
		\bigskip
		Now, we define a map $\mathfrak{K}$ and $\mathfrak{L} \otimes \mathcal{A}(\mathfrak{L},\mathfrak{K})$ by 
		\begin{equation}\label{MAP_LYA_morphism}
			\Phi_{\mathfrak{K}} : \mathfrak{K} \to \mathfrak{L} \otimes \mathcal{A}(\mathfrak{L},\mathfrak{K}), \quad \Phi_{\mathfrak{K}}(f_i) := \sum_{s=1}^n e_s \otimes x_{si} \quad \forall~~ i \in I.
		\end{equation}
		The following calculations show that the map $\Phi_\mathfrak{K}$ is a Lie-Yamaguti algebra morphism. 
		\begin{eqnarray*}
			[\Phi_{\mathfrak{K}}(f_i), \Phi_{\mathfrak{K}}(f_j)]_{\mathfrak{L} \otimes \mathcal{A}(\mathfrak{L},\mathfrak{K})} 
			&=& \left[\sum_{s=1}^{n} e_s \otimes x_{si}, \sum_{t=1}^n e_t \otimes x_{tj}\right]_{\mathfrak{L} \otimes \mathcal{A}(\mathfrak{L},\mathfrak{K})} \\
			&=& \sum_{s,t = 1}^{n} [e_s,e_t]_{\mathfrak{L}} \otimes x_{si}x_{tj} 
			~=~ \sum_{a = 1}^{n} e_a \otimes \left( \sum_{s,t = 1}^n \tau_{st}^a x_{si} x_{tj}\right) \\
			&=& \sum_{a = 1}^{n} e_a \otimes \left( \sum_{u \in B_{ij}} \alpha_{ij}^u x_{au} \right)
			~=~ \sum_{u \in B_{ij}} \alpha_{ij}^u \Phi_\mathfrak{K}(f_u) \\ \\ 
			&=& \Phi_\mathfrak{K} [f_i,f_j]_{\mathfrak{K}} ,
		\end{eqnarray*}
		for any $i,j \in I$. Moreover,
		\begin{eqnarray*}
			\{\Phi_{\mathfrak{K}}(f_i), \Phi_{\mathfrak{K}}(f_j), \Phi_{\mathfrak{K}}(f_k)\}_{\mathfrak{L} \otimes \mathcal{A}(\mathfrak{L},\mathfrak{K})} 
			&=& \left\{\sum_{r=1}^{n} e_r \otimes x_{ri}, \sum_{s=1}^{n} e_s \otimes x_{sj}, \sum_{t=1}^n e_t \otimes x_{tk}\right\} \\
			&=& \sum_{r,s,t = 1}^n \{e_r, e_s, e_t\}_{\mathfrak{L}} \otimes x_{ri} x_{sj} x_{tk} \\
			&=& \sum_{a = 1}^{n} e_a \otimes \left( \sum_{r,s,t = 1}^n \omega_{rst}^a x_{ri} x_{sj} x_{tk}\right) \\
			&=& \sum_{a = 1}^n e_a \otimes \left( \sum_{u \in C_{ijk}} \beta_{ijk}^u x_{au} \right) 
			~=~ \sum_{u \in C_{ijk}} \beta_{ijk}^u \Phi_\mathfrak{K}(f_u)\\ \\
			&=& \Phi_\mathfrak{K}\{f_i,f_j,f_k\}_{\mathfrak{K}},
		\end{eqnarray*}
		for any $i,j,k \in I$.
		Therefore, $\Phi_{\mathfrak{K}} : \mathfrak{K} \to \mathfrak{L} \otimes \mathcal{A}(\mathfrak{L},\mathfrak{K})$ is a Lie-Yamaguti algebra morphism. 
		
		\bigskip
		Now, for an arbitrary commutative algebra $A$, we define a map
		\begin{equation}\label{MAP_hom_set_bijection}
			\Psi_{\mathfrak{K},A} : 	\operatorname{Hom}_{\operatorname{ComAlg}_\mathbb{K}}(\mathcal{A}(\mathfrak{L},\mathfrak{K}), A) \longrightarrow \operatorname{Hom}_{\operatorname{LYA}_\mathbb{K}}(\mathfrak{K}, \mathfrak{L} \otimes A), \quad 
			\Psi_{\mathfrak{K},A}(\theta) := 	(id_\mathfrak{L} \otimes \theta) \circ \Phi_\mathfrak{K}\text{, i.e.,}
		\end{equation} 
		\begin{equation*}
			\begin{tikzcd}
				{\mathcal{A}(\mathfrak{L},\mathfrak{K})} & {\mathfrak{K}} 
				\\
				A & {\mathfrak{L} \otimes A.}
				\arrow[""{name=0, anchor=center, inner 	sep=0}, "\theta"', from=1-1, to=2-1]
				\arrow[""{name=1, anchor=center, inner 	sep=0}, "{(id_\mathfrak{L} \otimes \theta) \circ \Phi_\mathfrak{K}}", from=1-2, to=2-2]
				\arrow[shorten <=16pt, shorten >=16pt, 	maps to, from=0, to=1]
			\end{tikzcd}
		\end{equation*}
		We show that the map $\Psi_{\mathfrak{K}, A}$ is bijective. First, we ensure that $\Psi_{\mathfrak{K},A}$ is surjective, i.e., for any morphism $\gamma \in \operatorname{Hom}_{\operatorname{LYA}_\mathbb{K}}(\mathfrak{K}, \mathfrak{L} \otimes A)$ there exists a morphism 
		$\theta \in \operatorname{Hom}_{\operatorname{ComAlg}_\mathbb{K}}(\mathcal{A}(\mathfrak{L},\mathfrak{K}), A)$ for which the following diagram
		\begin{equation}\label{DIAG_univ_prop_of_A(L,-)}
			\begin{tikzcd}
				{\mathfrak{K}} & {\mathfrak{L} \otimes 	\mathcal{A}(\mathfrak{L} ,\mathfrak{K})} \\
				& {\mathfrak{L} \otimes A}
				\arrow["{\Phi_{\mathfrak{K}}}", 	from=1-1, to=1-2]
				\arrow["{id_{\mathfrak{L}} \otimes 	\theta}", from=1-2, to=2-2]
				\arrow["\gamma"', from=1-1, to=2-2]
			\end{tikzcd}
		\end{equation}
		commutes. Let $\gamma \in \operatorname{Hom}_{\operatorname{LYA}_\mathbb{K}}(\mathfrak{K}, \mathfrak{L} \otimes A)$, then for each basis element $f_i$ of $\mathfrak{K}$ (where $i \in I$), we have $\{g_{si} : s=1,2,\ldots,n, i \in I\} \subset A$  such that 
		\begin{equation}\label{EQN_LYA_thm1_gamma}
			\gamma(f_i) = \sum_{s=1}^n e_s \otimes 	g_{si}.
		\end{equation}
		Thus, we have the following expressions
		\begin{align*}
			\gamma[f_i,f_j]_{\mathfrak{K}} ~=~ 	\gamma\left( \sum_{u \in B_{ij}} \alpha_{ij}^u f_u\right) ~=~ \sum_{u \in B_{ij}} \alpha_{ij}^u \gamma(f_u) 
			~=~ \sum_{a = 1}^n e_a \otimes 	\left(\sum_{u \in B_{ij}} \alpha_{ij}^u g_{au}\right), 
		\end{align*} 
		\begin{align*}
			[\gamma(f_i), \gamma(f_j)]_{\mathfrak{L} 	\otimes A} 
			&~=~ \left[ \sum_{s=1}^n e_s \otimes 	g_{si}, \sum_{t=1}^n e_t \otimes g_{tj}\right] 
			~=~ \sum_{s,t = 1}^n [e_s,e_t] \otimes 	g_{si}g_{tj}  \nonumber \\
			&~=~ \sum_{s,t = 1}^n \left(\sum_{a = 1}^n 	\tau_{st}^a e_a \otimes g_{si}g_{tj}\right) ~=~ \sum_{a = 1}^n e_a \otimes \left( \sum_{s,t = 1}^n \tau_{st}^a g_{si} g_{tj}\right).
		\end{align*}
		for all $i,j \in I$. Also, 
		\begin{align*}
			\gamma \{f_i,f_j,f_k\}_{\mathfrak{K}} ~=~ 	\gamma&\left( \sum_{u \in C_{ijk}} \beta_{ijk}^u f_u\right) ~=~ \sum_{u \in C_{ijk}} \beta_{ij}^u \gamma(f_u) 
			~=~ \sum_{a = 1}^n e_a \otimes 	\left(\sum_{u \in C_{ijk}} \beta_{ij}^u g_{au}\right),
		\end{align*}
		\begin{align*} 
			\{\gamma(f_i), \gamma(f_j), 	\gamma(f_k)\}_{\mathfrak{L} \otimes A}  
			&~=~ \left\{\sum_{r=1}^n e_r \otimes 	g_{ri}, \sum_{s=1}^n e_s \otimes g_{sj}, \sum_{t=1}^n e_t \otimes g_{tk} \right\} \nonumber \\
			&~=~ \sum_{r,s,t = 1}^n \{e_r,e_s,e_t\} 	\otimes g_{ri}g_{sj}g_{tk} \nonumber \\
			&~=~ \sum_{r,s,t = 1}^n \left(\sum_{a = 	1}^n \omega_{rst}^a e_a \otimes g_{ri}g_{sj}g_{tk}\right) 
			~=~ \sum_{a = 1}^n e_a \otimes \left( 	\sum_{r,s,t = 1}^n \omega_{rst}^a g_{ri}g_{sj}g_{tk} \right),
		\end{align*}
		for all $i,j,k \in I$.
		Since $\gamma: \mathfrak{K} \to \mathfrak{L} \otimes A$ is a Lie-Yamaguti algebra morphism, it follows that the family of elements $\{g_{si} : s=1,2,\ldots,n, i \in I\}$ must satisfy the following identities in $A$. 
		\begin{equation}\label{EQN_struct_const_alpha_tau_relation}
			\sum_{u \in B_{ij}} \alpha_{ij}^u g_{au} 	~=~ \sum_{s,t = 1}^n \tau_{st}^a g_{si}g_{tj}, \quad	\forall~ i,j \in I~~\text{and}~~ a=1,2,\ldots,n,
		\end{equation}
		\begin{equation}\label{EQN_struct_const_bcommutes.eta_omega_relation}
			\sum_{u \in C_{ijk}} \beta_{ijk}^u g_{au} 	~=~ \sum_{r,s,t = 1}^n \omega_{rst}^a g_{ri}g_{sj}g_{tk},\quad \forall~ i,j,k \in I~~\text{and}~~ a=1,2,\ldots,n.
		\end{equation}
		Now, the universal property of the polynomial algebra yields a unique algebra homomorphism $\mu:\mathbb{K}[X_{si} : s = 1, \ldots, n, i \in I] \longrightarrow A$ such that $\mu(X_{si}) = g_{si}$, for all $s=1,2,\ldots,n$ and $i \in I$. We have $J \subset \operatorname{Ker}(\mu)$ since 
		\[\mu\left(P_{(a,i,j)}^{(\mathfrak{L},\mathfrak{K})}\right) ~=~ \mu\left( \sum_{u \in B_{ij}} \alpha_{ij}^u X_{au} - \sum_{s,t = 1}^n \tau_{st}^a X_{si}X_{tj}\right) ~=~ \sum_{u \in B_{ij}} \alpha_{ij}^u g_{au} - \sum_{s,t = 1}^n \tau_{st}^a g_{si}g_{tj} ~=~ 0,\]
		and
		\begin{align*}
			\mu\left(Q_{(a,i,j,k)}^{(\mathfrak{L},\mathfrak{K})}\right) ~=~ \mu \left(\sum_{u \in C_{ijk}} \beta_{ijk}^u X_{au} - \sum_{r,s,t = 1}^n \omega_{rst}^{a} X_{ri} X_{sj} X_{tk}\right) 
			&~=~ \sum_{u \in C_{ijk}} \beta_{ijk}^u g_{au} - \sum_{r,s,t = 1}^n \omega_{rst}^a g_{ri}g_{sj}g_{tk}
			~=~ 0,
		\end{align*}
		for any $i,j,k \in I$ and $a = 1,\ldots, n$.	
		Hence, there exists a unique algebra morphism $\theta : \mathcal{A}(\mathfrak{L},\mathfrak{K}) \to A$ such that $\theta(x_{si}) = g_{si}$, for all $s=1,2,\ldots,n$ and $i \in I$. 
		\[\begin{tikzcd}
			{k[X_{si} : s=1,\ldots,n, i \in I]} && A \\
			\\
			{\mathcal{A}(\mathfrak{L},\mathfrak{K})}
			\arrow["\mu", from=1-1, to=1-3]
			\arrow["{(\text{quotient map}) ~\pi}"', 	from=1-1, to=3-1]
			\arrow["{\exists! ~\theta }"', dotted, 	from=3-1, to=1-3].
		\end{tikzcd}\]
		Thus, for any $i \in I$ we note that
		\[
			(id_{\mathfrak{L}} \otimes \theta) \circ \Phi_{\mathfrak{K}}(f_i) ~=~ (id_{\mathfrak{L}} \otimes \theta) \left(\sum_{s=1}^n e_s \otimes x_{si}\right) ~=~ \sum_{s=1}^n e_s \otimes g_{si} ~=~ \gamma(f_i).
		\]
		It implies that $(id_{\mathfrak{L}} \otimes \theta) \circ \Phi_{L} = \gamma$. Hence, the map $\Psi_{\mathfrak{K},A}$ surjective. 
		
		\medskip\noindent
		To show $\Psi_{\mathfrak{K},A}$ is injective, we show that $\theta$ is the unique such morphism for which diagram (\ref{DIAG_univ_prop_of_A(L,-)}) commutes. Let $\widetilde{\theta} : \mathcal{A}(\mathfrak{L},\mathfrak{K}) \to A$ be another algebra morphism such that $(id_{\mathfrak{L}} \otimes \widetilde{\theta}) \circ \Phi_{\mathfrak{K}}(f_i) = \gamma(f_i)$ for all $i \in I$. Then 
		\begin{equation*}
			\sum_{s=1}^n e_s \otimes \widetilde{\theta}(x_{si}) = \sum_{s=1}^n e_s \otimes g_{si}.
		\end{equation*} 
		Therefore $\widetilde{\theta}(x_{si}) = g_{si} = \theta(x_{si})$, for all $s=1,2,\ldots,n$ and $i \in I$. Since the set $\{x_{si} : s = 1,2,\ldots,n, i \in I\}$ generates the algebra $\mathcal{A}(\mathfrak{L},\mathfrak{K})$, we obtain $\widetilde{\theta} = \theta$. Thus, proving $\Psi_{\mathfrak{K},A} : \operatorname{Hom}_{\operatorname{ComAlg}_\mathbb{K}}(\mathcal{A}(\mathfrak{L},\mathfrak{K}), A) \longrightarrow \operatorname{Hom}_{\operatorname{LYA}_\mathbb{K}}(\mathfrak{K}, \mathfrak{L} \otimes A)$ is a bijection. 
		
		\medskip
		Finally, we show that $\mathcal{A}(\mathfrak{L},-) : \operatorname{LYA}_\mathbb{K} \longrightarrow \operatorname{ComAlg}_\mathbb{K}$ is a functor which is left adjoint to the functor $\mathfrak{L} \otimes - : \operatorname{ComAlg}_{\mathbb{K}} \to \operatorname{LYA}_\mathbb{K}$. 
		Let $\mathfrak{K}, \mathfrak{K'} \in \operatorname{LYA}_\mathbb{K}$ and $\zeta:\mathfrak{K} \to \mathfrak{K'}$ be a Lie-Yamaguti algebra morphism. Using the bijectivity of $\Psi_{\mathfrak{K},A}$ for the Lie-Yamaguti algebra morphism $\Phi_{\mathfrak{K'}} \circ \zeta : \mathfrak{K} \to \mathfrak{L} \otimes \mathcal{A}(\mathfrak{L},\mathfrak{K'})$, we have a unique algebra morphism $\theta : \mathcal{A}(\mathfrak{L},\mathfrak{K}) \to \mathcal{A}(\mathfrak{L},\mathfrak{K'})$ such that the following diagram
		\begin{equation}\label{DIAG_functoriality_of_A(L,-)}
			\begin{tikzcd}
				{\mathfrak{K}} & {\mathfrak{L} 	\otimes \mathcal{A}(\mathfrak{L} , \mathfrak{K})}
				\\
				{\mathfrak{K'}} & {\mathfrak{L} 	\otimes \mathcal{A}(\mathfrak{L} , \mathfrak{K'})}
				\arrow["\zeta"', from=1-1, to=2-1]
				\arrow["{id_{\mathfrak{L}} \otimes 	\theta}", from=1-2, to=2-2]
				\arrow["{\Phi_{\mathfrak{K}}}", 	from=1-1, to=1-2]
				\arrow["{\Phi_{\mathfrak{K'}}}", 	from=2-1, to=2-2]
			\end{tikzcd}
		\end{equation}
		commutes. We denote this unique morphism $\theta$ by $\mathcal{A}(\mathfrak{L},\zeta)$. With this assignment, $\mathcal{A}(\mathfrak{L},-) : \operatorname{LYA}_\mathfrak{K} \to \operatorname{ComAlg}_\mathfrak{K}$ becomes a functor. It follows that the diagram
		
		\begin{equation*}
			\begin{tikzcd}
				{\operatorname{Hom}_{\operatorname{ComAlg}_{\mathbb{K}}}(\mathcal{A}(\mathfrak{L},\mathfrak{K}), A)} & {\operatorname{Hom}_{\operatorname{ComAlg}_{\mathbb{K}}}(\mathcal{A}(\mathfrak{L},\mathfrak{K'}), A)} & {\operatorname{Hom}_{\operatorname{ComAlg}_{\mathbb{K}}}(\mathcal{A}(\mathfrak{L},\mathfrak{K'}), A')} \\
				{\operatorname{Hom}_{\operatorname{LYA}_{\mathbb{K}}}(\mathfrak{K},\mathfrak{L}\otimes A)} & {\operatorname{Hom}_{\operatorname{LYA}_{\mathbb{K}}}(\mathfrak{K'},\mathfrak{L}\otimes A)} & {\operatorname{Hom}_{\operatorname{LYA}_{\mathbb{K}}}(\mathfrak{K'},\mathfrak{L}\otimes A')}
				\arrow["{\Psi_{\mathfrak{K},A}}"', from=1-1, to=2-1]
				\arrow["{\Psi_{\mathfrak{K'},A}}"', from=1-2, to=2-2]
				\arrow["{\Psi_{\mathfrak{K'},A'}}", from=1-3, to=2-3]
				\arrow[from=1-1, to=1-2]
				\arrow[from=1-2, to=1-3]
				\arrow[from=2-1, to=2-2]
				\arrow[from=2-2, to=2-3]
			\end{tikzcd}
		\end{equation*}
	
		\smallskip\noindent
	 	commutes, showing the naturality of $\Psi_{\mathfrak{K},A}$ in both $\mathfrak{K}$ and $A$. Hence, to conclude, $\mathcal{A}(\mathfrak{L},-)$ is left adjoint functor of the current Lie-Yamaguti algebra functor $\mathfrak{L} \otimes -$. 
	\end{proof}

	\begin{remark}
		The set of all algebra morphisms from  $\mathcal{A}(\mathfrak{L},\mathfrak{K})$ to $\mathbb{K}$ parametrize the set of all Lie-Yamaguti algebra morphism from $\mathfrak{K}$ to $\mathfrak{L}$. This observation follows as a straightforward consequence of the bijection described in (\ref{MAP_hom_set_bijection}) by taking $A = \mathbb{K}$.
	\end{remark}
	
	\begin{corollary}
		\label{COR_LYA_hom_set_bijection}
		Let $\mathfrak{L}$ and $\mathfrak{K}$ be two Lie-Yamaguti algebras with $\mathfrak{L}$ being finite-dimensional. Moreover, let $A$ be a commutative algebra. Then the following map is bijective
		\begin{equation}
			\Psi_{\mathfrak{K},A} : 
			\operatorname{Hom}_{\operatorname{Alg}_\mathbb{K}}(\mathcal{A}(\mathfrak{L},\mathfrak{K}), A) \longrightarrow \operatorname{Hom}_{\operatorname{LYA}_\mathbb{K}}(\mathfrak{K}, \mathfrak{L} \otimes A), \quad \Psi_{\mathfrak{K},A}(\theta) := (id_\mathfrak{L} \otimes \theta) \circ \Phi_\mathfrak{K}.
		\end{equation}
	\end{corollary}
	
	\begin{example} 
		If $\mathfrak{L}$ and $\mathfrak{K}$ are abelian Lie-Yamaguti algebras then 
		$\mathcal{A}(\mathfrak{L}, \mathfrak{K}) \cong \mathbb{K}[X_{si} : s=1,\ldots,n, i \in I]$, where $n = dim_{\mathbb{K}}(\mathfrak{L})$ and $|I| = dim_{\mathbb{K}}(\mathfrak{K})$.
	\end{example}

	\begin{example}
		Let $\mathfrak{L}$ be an $n$-dimensional Lie-Yamaguti algebra with structure constants $\{\tau_{ij}^s: i,j,s=1,2,\ldots,n\} \cup \{\omega_{ijk}^s: i,j,k,s=1,2,\ldots,n\}$. Then $\mathcal{A}(\mathfrak{L},\mathbb{K}) \cong \mathbb{K}[X_1,X_2,\ldots,X_n]/J$, where $J$ is the ideal generated by the polynomials $\sum_{s,t=1}^n \tau_{s,t}^a X_s X_t$ and $\sum_{r,s,t=1}^n \omega_{rst}^a X_r X_s X_t$ for all $a=1,\ldots,n$.
 	\end{example}
 
 	\begin{example}
 		Let $\mathfrak{L}$ be an $n$-dimensional Lie-Yamaguti algebra with structure constants $\{\tau_{ij}^s: i,j,s=1,2,\ldots,n\} \cup \{\omega_{ijk}^s: i,j,k,s=1,2,\ldots,n\}$. Then $\mathcal{A}(\mathbb{K},\mathfrak{L}) \cong \mathbb{K}[X_1,X_2,\ldots,X_n]/J$, where $J$ is the ideal generated by the polynomials
 		\begin{equation*}
 			\sum_{u=1}^n \tau_{ij}^u X_{u} \quad \text{for all}~ i,j =1,2,\ldots,n; \quad \text{and} \quad \sum_{u=1}^n \omega_{ijk}^u X_{u} \quad \text{for all}~ i,j,k = 1,2,\ldots,n.
 		\end{equation*}
 		Now, we recall that the derived subalgebra $\mathfrak{L'}$ of $\mathfrak{L}$ is defined as   
 		\begin{equation*}
 			\mathfrak{L'} := \big[\mathfrak{L},\mathfrak{L}\big] + \big\{\mathfrak{L}, \mathfrak{L}, \mathfrak{L}\big\}
 		\end{equation*}
 		and the fact that for any vector space $V$ with a basis $B$, the symmetric algebra $S(V)$ is isomorphic to the polynomial ring $\mathbb{K}[B]$ where elements of $B$ are considered as indeterminates. 
 		Since the polynomials of the form $\sum_{u=1}^n \tau_{ij}^uX_u$ corresponds to the elements of $[\mathfrak{L}, \mathfrak{L}]$ and polynomials of form $\sum_{u=1}^n \omega_{ijk}^u X_{u}$ corresponds to the elements of $\{\mathfrak{L}, \mathfrak{L}, \mathfrak{L}\}$, we use the isomorphism $S(V) \cong \mathbb{K}[B]$ to  conclude that symmetric algebra of $\mathfrak{L}/\mathfrak{L'}$ denoted by $S(\mathfrak{L}/\mathfrak{L'})$ is isomorphic to $\mathbb{K}[X_1,X_2,\ldots,X_n]/J$. Hence, we obtain 
 		\begin{equation*}
 			\mathcal{A}(\mathbb{K},\mathfrak{L}) \cong S(\mathfrak{L}/\mathfrak{L'}).
 		\end{equation*}  
 	\end{example}
	
	\begin{definition} \label{DEF_LYA_universal_algebra}
		Let $\mathfrak{L}$ and $\mathfrak{\mathfrak{K}}$ be two Lie-Yamaguti algebras with $\mathfrak{L}$ being finite dimensional. Then the commutative algebra $\mathcal{A}(\mathfrak{L},\mathfrak{K})$ is called the {\bf universal algebra} of $\mathfrak{L}$ and $\mathfrak{K}$. If $\mathfrak{L} = \mathfrak{K}$, then we denote the universal algebra of $\mathfrak{L}$ simply by $\mathcal{A}(\mathfrak{L})$.
	\end{definition}
	
	\begin{remark}
		Let $(\mathfrak{h},\cdot)$ be a Leibniz algebra. Then $\mathfrak{h}$ becomes a Lie-Yamaguti algebra with the following brackets
		\begin{equation}\label{EQN_leib_to_lie-yamaguti}
			[x,y] := x \cdot y - y \cdot x \quad \forall~ x,y \in \mathfrak{h}, \qquad \{x,y,z\} := -x \cdot (y \cdot z) \quad \forall~ x,y,z \in \mathfrak{h}.
		\end{equation}
		Furthermore, for any Leibniz algebra morphism $f: \mathfrak{h} \to \mathfrak{g}$, $f$ is also a Lie-Yamaguti algebra morphism. Therefore, giving a functor from $\operatorname{Leib}_{\mathbb{K}} \to \operatorname{LYA}_\mathbb{K}$. 
		The universal algebra $\mathcal{A}(\mathfrak{h})$ of the Lie-Yamaguti algebra $(\mathfrak{h},[\cdot,\cdot], \{\cdot,\cdot,\cdot\})$ is a quotient of the universal algebra of the Leibniz algebra $(\mathfrak{h},\cdot)$, constructed in \cite{agore20}.
	\end{remark}

	\begin{definition}
		\label{DEF_LYA_L_univ_poly}
		If $\{\tau_{ij}^s : i,j,s = 1,2,\ldots, n\}$ and $\{\omega_{ijk}^s : i,j,k,s = 1,2,\ldots,n\}$ are structure constants of $\mathfrak{L}$ where $n$ is the dimension of $\mathfrak{L}$. Then, the polynomials defined by
		\begin{equation}\label{DEF_LYA_L_univ_poly1}
			P_{(a,i,j)}^{(\mathfrak{L})} := \sum_{u = 	1}^n \tau_{ij}^u X_{au} - \sum_{s,t = 1}^n \tau_{st}^a X_{si} X_{tj} \in \mathbb{K}[X_{ij} : i,j = 1,\ldots,n], \quad \forall~~ a,i,j = 1,2,\ldots,n,
		\end{equation}
		\begin{equation}\label{DEF_LYA_L_univ_poly2}
			Q_{(a,i,j,k)}^{(\mathfrak{L})} := \sum_{u = 	1}^n \omega_{ijk}^u X_{au} - \sum_{r,s,t = 1}^n \omega_{rst}^a X_{ri} X_{sj} X_{tk} \in \mathbb{K}[X_{ij} : i,j = 1,\ldots,n], \quad \forall~~ a,i,j,k = 1,2,\ldots,n.
		\end{equation}
		are called the {\bf universal polynomials of $\mathfrak{L}$.}
	\end{definition}
	
	\noindent
	The universal algebra $\mathcal{A}(\mathfrak{L})$ satisfies the following universal property.   
	
	\begin{corollary}\label{COR_DIAG_univ_prop_of_A(L,-)}
		Let $\mathfrak{L}$ be a finite-dimensional Lie-Yamaguti algebra of dimension $n$. Then for any commutative algebra $A$ and any Lie-Yamaguti algebra morphism $\gamma: \mathfrak{L} \to \mathfrak{L} \otimes A$, there exists a unique algebra morphism $\theta: \mathcal{A}(\mathfrak{L}) \to A$ such that $\gamma = (id_\mathfrak{L} \otimes \theta) \circ \Phi_{\mathfrak{L}}$, i.e., the following diagram commutes.
		\begin{equation}
			\begin{tikzcd}
				{\mathfrak{L}} & {\mathfrak{L} \otimes 	\mathcal{A}(\mathfrak{L})} \\
				& {\mathfrak{L} \otimes A}
				\arrow["{\Phi_{\mathfrak{L}}}", 	from=1-1, to=1-2]
				\arrow["{id_{\mathfrak{L}} \otimes 	\theta }", from=1-2, to=2-2]
				\arrow["\gamma"', from=1-1, to=2-2]
			\end{tikzcd}
		\end{equation} 
	\end{corollary}

	\noindent
	Below we give an explicit description of the universal algebra for a family of Heisenberg Lie-Yamaguti algebras introduced in \cite{our-paper}. 
	
	\begin{example}
		A Heisenberg Lie-Yamaguti algebra is a vector space $\mathfrak{H}_n$ generated by $2n+1$ elements ($n \ge 1$) $$\{e_0,e_1,\ldots,e_n,e_{n+1},\ldots,e_{2n}\}.$$The Lie-Yamaguti brackets on $\mathfrak{H}_n$ are given by 
		\[
		[e_i,e_{n+i}] = e_0,\quad \{e_i,e_{n+i},e_i\} = e_0, \quad 1 \le i \le n.
		\]   
		All other brackets on the basis elements are either zero or are determined by the relations of Lie-Yamaguti algebra. Therefore, the structure constants of $\mathfrak{H}_n$, which are $\{\tau_{ij}^s: s,i,j = 0,1,\ldots,2n\}$ and $\{\omega_{ijk}^s: s,i,j,k = 0,1,\ldots,2n\}$, can be described as follows: 
		\begin{eqnarray*}
			\tau_{ij}^s &=& 
			\begin{cases}
				\hspace{2.5mm} 1 \qquad (s=0,~ 1 \le i \le n,~j=n+i); \\
				-1 \qquad (s=0,~ n+1 \le i \le 2n,~ j=i-n); \\
				\hspace{2.5mm} 0 ~\qquad \text{otherwise}.
			\end{cases} \\ \\
			\omega_{ijk}^s &=& 
			\begin{cases}
				\hspace{2.5mm} 1 \qquad (s=0,~ 1 \le i \le n,~ j=n+i,~ k=i);\\
				-1 \qquad (s=0,~ n+1 \le i \le 2n,~ j=i-n,~ k=i);\\
				\hspace{2.5mm} 0 ~\qquad \text{otherwise.}
			\end{cases}
		\end{eqnarray*}
		
		\smallskip\noindent
		Hence, the universal polynomials of $\mathfrak{H}_n$ of the form $P_{(a,i,j)}^{(\mathfrak{H}_n)}$ for $a = 1,2,\ldots,2n$ are given by: 
		\begin{equation*}
			P_{(a,i,j)}^{(\mathfrak{H}_n)} = 
			\begin{cases}
				\hspace{3.0mm} X_{a0} \qquad (1 \le i \le n,~ j = n+i) ; \\
				-X_{a0} \qquad (n+1 \le i \le 2n,~ j = i-n); \\ 
				\hspace{5mm} 0 \hspace{2.5mm}\quad\quad \text{otherwise.} 
			\end{cases}
		\end{equation*}
		and for any $a = 0$ it is given by: 
		\begin{equation*}
			P_{(a,i,j)}^{(\mathfrak{H}_n)} = 
			\begin{cases}
				\hspace{3.0mm} X_{00} 
				- \sum_{s=1}^n X_{si}X_{(n+s)j} + \sum_{s=n+1}^{2n} X_{si}X_{(n-s)j}\qquad (1 \le i \le n,~ j = n+i) ; \\
				-X_{00} - \sum_{s=1}^n X_{si}X_{(n+s)j} + \sum_{s=n+1}^{2n} X_{si}X_{(n-s)j} \qquad (n+1 \le i \le 2n,~ j = i-n); \\ 
				- \sum_{s=1}^n X_{si}X_{(n+s)j} + \sum_{s=n+1}^{2n} X_{si}X_{(n-s)j} \hspace{2.8mm}\qquad\qquad \text{otherwise.} 
			\end{cases}
		\end{equation*}
		Furthermore, the universal polynomials of $\mathfrak{H}_n$ of the form $Q_{(a,i,j,k)}^{(\mathfrak{H}_n)}$ for $a=1,2,\ldots,2n$ are given by: 
		\begin{equation*}
			Q_{(a,i,j,k)}^{(\mathfrak{H}_n)} = 
			\begin{cases}
				\hspace{3.0mm} X_{a0} \qquad (1 \le i \le n,~ j = n+i,~ k=i) ; \\
				-X_{a0} \qquad (n+1 \le i \le 2n,~ j = i-n,~ k=i); \\ 
				\hspace{5mm} 0 \hspace{2.5mm}\quad\quad \text{otherwise.} 
			\end{cases}
		\end{equation*}
		and for any $a = 0$ it is given by: 
		\begin{equation*}
			Q_{(a,i,j,k)}^{(\mathfrak{H}_n)} = 
			\begin{cases}
				\hspace{3.0mm} X_{00} 
				- \sum_{r=1}^n X_{ri}X_{(n+r)j}X_{rk} + \sum_{r=n+1}^{2n} X_{ri}X_{(n-r)j}X_{rk} \qquad (1 \le i \le n,~ j = n+i,~ k = i) ; \\
				-X_{00} 
				- \sum_{r=1}^n X_{ri}X_{(n+r)j}X_{rk} + \sum_{r=n+1}^{2n} X_{ri}X_{(n-r)j}X_{rk} \qquad (n+1 \le i \le 2n,~ j = i-n,~ k=i); \\ 
				- \sum_{r=1}^n X_{ri}X_{(n+r)j}X_{rk} + \sum_{r=n+1}^{2n} X_{ri}X_{(n-r)j}X_{rk} \hspace{6.5mm}\qquad\quad 
				\text{otherwise.}
			\end{cases}
		\end{equation*}
		
		\medskip \noindent	
		Therefore, the universal algebra of $\mathfrak{H}_n$ is given by   
		\begin{equation*}
			\mathcal{A}(\mathfrak{H}_n) \cong \mathbb{K}[X_{ij}: 0 \le i,j \le 2n] \big/ J 
		\end{equation*}
		where $J$ is the ideal generated by the universal polynomials of $\mathfrak{H}_n$ described above. 
 	\end{example}
	
	\bigskip
	\subsection*{\large Bialgebra structure on the universal algebra $\mathcal{A}(\mathfrak{L})$.} ~
	
	\bigskip\noindent
	The universal algebra $\mathcal{A}(\mathfrak{L})$ of a finite-dimensional Lie-Yamaguti algebra $\mathfrak{L}$ admits a bialgebra structure. 
	\begin{theorem}
		Let $\mathfrak{L}$ be a Lie-Yamaguti algebra of dimension $n$. Then there exists a unique bialgebra structure on $\mathcal{A}(\mathfrak{L})$ such that the Lie-Yamaguti algebra morphism $\Phi_{\mathfrak{L}}: \mathfrak{L} \to \mathfrak{L} \otimes \mathcal{A}(\mathfrak{L})$ gives a right $\mathcal{A}(\mathfrak{L})$-comodule structure on $\mathfrak{L}$. Where, the comultiplication and the counit map on $\mathcal{A}(\mathfrak{L})$ are given by 
		\begin{equation}\label{DEF_comult_counit_A(L)}
			\Delta(x_{ij}) = \sum_{s=1}^n x_{is} 	\otimes x_{sj} ~~\text{and}~~ \varepsilon(x_{ij}) = \delta_{ij}, \quad  \forall~~ i,j =1,\ldots,n. 
		\end{equation}
	\end{theorem}
	
	\begin{proof}
		Note that $\mathcal{A}(\mathfrak{L}) \otimes \mathcal{A}(\mathfrak{L})$ is also an algebra. Consider the following Lie-Yamaguti algebra morphism $f: \mathfrak{L} \to \mathfrak{L} \otimes \mathcal{A}(\mathfrak{L}) \otimes \mathcal{A}(\mathfrak{L})$ given by $f := (\Phi_{\mathfrak{L}}\otimes id_{\mathcal{A}(\mathfrak{L})}) \circ \Phi_{\mathfrak{L}}$. It follows from Corollary \ref{COR_DIAG_univ_prop_of_A(L,-)} that there exists a unique algebra morphism $\Delta: \mathcal{A}(\mathfrak{L}) \to \mathcal{A}(\mathfrak{L}) \otimes \mathcal{A}(\mathfrak{L})$ such that $(id_{\mathfrak{L}} \otimes \Delta) \circ \Phi_{\mathfrak{L}} = f$. In other words, the following diagram 
		\begin{equation}\label{DIAG_bialgebra_A(L)_comult}
			\begin{tikzcd}
				{\mathfrak{L}} & {\mathfrak{L} \otimes 	\mathcal{A}(\mathfrak{L})} \\
				{\mathfrak{L} \otimes 	\mathcal{A}(\mathfrak{L})} 	& {\mathfrak{L} \otimes \mathcal{A}(\mathfrak{L}) \otimes \mathcal{A}(\mathfrak{L})}
				\arrow["{\Phi_{\mathfrak{L}}}", 	from=1-1, to=1-2]
				\arrow["{\Phi_{\mathfrak{L}} \otimes 		id_{\mathcal{A}(\mathfrak{L})}}", from=2-1, to=2-2]
				\arrow["{id_{\mathfrak{L}} \otimes 	\Delta}", 	from=1-2, to=2-2]
				\arrow["{\Phi_{\mathfrak{L}}}"', 	from=1-1, to=2-1]
			\end{tikzcd}
		\end{equation}
		
		\medskip \noindent 
		commutes. Evaluating the above diagram at each $e_i$, for $i=1,\ldots,n$, we get
		\[\sum_{t=1}^n e_t \otimes \Delta(x_{ti}) ~=~(\Phi_{\mathfrak{L}} \otimes id_{\mathcal{A}(\mathfrak{L})})\left(\sum_{s=1}^n e_s \otimes x_{si}\right) ~=~ \sum_{s=1}^n \left(\sum_{t=1}^n e_t \otimes x_{ts}\right) \otimes x_{si} ~=~ \sum_{t=1}^n e_t \otimes \left( \sum_{s=1}^n x_{ts} \otimes x_{si}\right).\]
		This gives us $\Delta (x_{ti}) ~=~ \sum_{s=1}^n x_{ts} \otimes x_{si}$ for all $t,i =1,\ldots, n$. The map $\Delta$ is a coassociative coproduct. Again by Corollary \ref{COR_DIAG_univ_prop_of_A(L,-)}, there exists a unique algebra morphism $\varepsilon: \mathcal{A}(\mathfrak{L}) \to \mathbb{K}$ such that the diagram
		\begin{equation}\label{DIAG_bialgebra_A(L)_counit}
			\begin{tikzcd}
				{\mathfrak{L}} & {\mathfrak{L} \otimes 	\mathcal{A}(\mathfrak{L})}
				\\
				& {\mathfrak{L} \otimes \mathbb{K}}
				\arrow["{\Phi_{\mathfrak{L}}}", 	from=1-1, to=1-2]
				\arrow["{id_{\mathfrak{L}} \otimes 	\varepsilon}", 	from=1-2, to=2-2]
				\arrow[from=1-1, to=2-2]
			\end{tikzcd}
		\end{equation}
		commutes. Here $\mathfrak{L} \to \mathfrak{L} \otimes \mathbb{K}$ is the canonical isomorphism given by $x \mapsto x \otimes 1$ for all $x \in \mathfrak{L}$. Evaluating the diagram at each $e_t$, $t =1,\ldots,n$, we obtain $\varepsilon (x_{ij}) = \delta_{ij}$ for all $i,j =1,\ldots,n$. Verifying that $\varepsilon $ is a count for $\Delta$ is easy. Thus, $\mathcal{A}(\mathfrak{L})$ is a bialgebra.
		
		\medskip
		The commutativity of diagrams (\ref{DIAG_bialgebra_A(L)_comult}) and (\ref{DIAG_bialgebra_A(L)_counit}) implies that the map $\Phi_{\mathfrak{L}}:\mathfrak{L} \to \mathfrak{L} \otimes \mathcal{A}(\mathfrak{L})$ defines a right $\mathcal{A}(\mathfrak{L})$-comodule structure on $\mathfrak{L}$.
	\end{proof} 
	
	\begin{definition}
		Let $\mathfrak{L}$ be a Lie-Yamaguti algebra of dimension $n$. We call the pair $(\mathcal{A}(\mathfrak{L}), \Phi_{\mathfrak{L}})$ the {\bf universal coacting bialgebra} of the Lie-Yamaguti algebra $\mathfrak{L}$.
	\end{definition}
	
	\noindent
	The pair $(\mathcal{A}(\mathfrak{L}), \Phi_{\mathfrak{L}})$  satisfies the following universal property which extends Corollary \ref{COR_DIAG_univ_prop_of_A(L,-)}.
	
	\begin{theorem}\label{THM_bialgebra_A(L)_univ_prop}
		Let $\mathfrak{L}$ be a Lie-Yamaguti algebra of dimension $n$. Then for any commutative bialgebra $B$ and any Lie-Yamaguti algebra morphism $f: \mathfrak{L} \to \mathfrak{L} \otimes B$ making $\mathfrak{L}$ a right $B$-comodule, there exists a unique bialgebra morphism $\theta: \mathcal{A}(\mathfrak{L}) \to B$ such that the following diagram commutes
		\begin{equation}\label{DIAG_bialgebra_A(L)_univ_prop}
			\begin{tikzcd}
				{\mathfrak{L}} & {\mathfrak{L} \otimes 		\mathcal{A}(\mathfrak{L})} \\
				& {\mathfrak{L} \otimes B}
				\arrow["{\Phi_{\mathfrak{L}}}", 	from=1-1, to=1-2]
				\arrow["{id \otimes \theta }", 	from=1-2, to=2-2]
				\arrow["f"', from=1-1, to=2-2].
			\end{tikzcd}
		\end{equation} 
	\end{theorem}
	
	\begin{proof}
		As $\mathcal{A}(\mathfrak{L})$ is the universal algebra of $\mathfrak{L}$, there exists a unique algebra morphism $\theta: \mathcal{A}(\mathfrak{L}) \to B$ such that diagram (\ref{DIAG_bialgebra_A(L)_univ_prop}) commutes. If we show that $\theta: \mathcal{A}(\mathfrak{L}) \to B$ is also a coalgebra morphism, then we are done. For that we need to show the following diagrams
		\begin{equation}\label{DIAG_bialgebra_A(L)_univ_prop_1}
			\begin{tikzcd}
				\mathcal{A}(\mathfrak{L}) & B && 		\mathcal{A}(\mathfrak{L}) & B \\
				{\mathcal{A}(\mathfrak{L}) \otimes 		\mathcal{A}(\mathfrak{L})} & {B \otimes B} &&& {\mathbb{K}}
				\arrow["\theta", from=1-1, to=1-2]
				\arrow["{\theta \otimes \theta}", 	from=2-1, 	to=2-2]
				\arrow["\Delta"', from=1-1, to=2-1]
				\arrow["{\Delta_B}", from=1-2, to=2-2]
				\arrow["\theta", from=1-4, to=1-5]
				\arrow["{\varepsilon_B}", from=1-5, 	to=2-5]
				\arrow["\varepsilon"', from=1-4, to=2-5]
			\end{tikzcd}
		\end{equation}
		are commutative. To show the commutativity of the first diagram, we use the universal property of $\mathcal{A}(\mathfrak{L})$. Considering $B \otimes B$ as a commutative algebra, there exists a unique algebra morphism $\psi: \mathcal{A}(\mathfrak{L}) \to B \otimes B$ such that the following diagram 
		\begin{equation}\label{DIAG_bialgebra_A(L)_univ_prop_2}
			\begin{tikzcd}
				{\mathfrak{L}} && {\mathfrak{L} \otimes
				\mathcal{A}(\mathfrak{L})} \\
				{\mathfrak{L} \otimes 	\mathcal{A}(\mathfrak{L})} && {\mathfrak{L} \otimes B \otimes B}
				\arrow["{id_{\mathfrak{L}} \otimes 	\psi}", 	from=1-3, to=2-3]
				\arrow["{\Phi_{\mathfrak{L}}}"', 	from=1-1, 	to=2-1]
				\arrow["{id_{\mathfrak{L}} \otimes 	(\Delta_{B} 	\circ \theta)}", from=2-1, to=2-3]
				\arrow["{\Phi_{\mathfrak{L}}}", 	from=1-1, 	to=1-3].
			\end{tikzcd}
		\end{equation}
		is commutative (Note that since $\theta$ and $\Delta_B$ are both algebra morphisms, we have $\Delta_B \circ \theta : \mathcal{A}(\mathfrak{L}) \to B \otimes B$ as an algebra morphism. Hence, making $id_{\mathfrak{L}} \otimes (\Delta_B \circ \theta) \circ \Phi_{\mathfrak{L}} : \mathfrak{L} \to \mathfrak{L} \otimes B \otimes B$ a Lie-Yamaguti algebra morphism). 
		
		\medskip
		Now, we show that the algebra morphism $(\theta \otimes \theta) \circ \Delta: \mathcal{A}(\mathfrak{L}) \to B \otimes B$ is a choice for $\psi$ by showing 
		$$\Omega := (id_{\mathfrak{L}} \otimes ((\theta \otimes \theta) \circ \Delta)) \circ \Phi_{\mathfrak{L}}  ~=~ (id_{\mathfrak{L}} \otimes (\Delta_B \circ \theta)) \circ \Phi_{\mathfrak{L}}.$$
		Since $f: \mathfrak{L} \to \mathfrak{L} \otimes B$ gives a right $B$-comodule structure on $\mathfrak{L}$ and for any morphisms $f,f',g,g'$ we have $(f \otimes g) \circ (f' \otimes g') = (f \circ f') \otimes (g \circ g')$, we obtain the following identity 
		\begin{eqnarray*}
			(id_{\mathfrak{L}} \otimes ((\theta \otimes 	\theta) \circ \Delta)) \circ \Phi_{\mathfrak{L}} 
			&=& (id_{\mathfrak{L}} \otimes \theta 	\otimes \theta) \circ \underbrace{(id_{\mathfrak{L}} \otimes \Delta) \circ \Phi_{\mathfrak{L}}}\\
			&\stackrel{(\ref{DIAG_bialgebra_A(L)_comult})}{=}& (id_{\mathfrak{L}} \otimes \theta \otimes \theta) \circ (\Phi_{\mathfrak{L}} \otimes id_{\mathcal{A}({\mathfrak{L}})}) \circ \Phi_{\mathfrak{L}}\\
			&=& (\underbrace{(id_{\mathfrak{L}} \otimes 	\theta \circ \Phi_{\mathfrak{L}})} \otimes (\theta 
			\circ id_{\mathcal{A}(\mathfrak{L})})) 	\circ \Phi_{\mathfrak{L}}\\
			&\stackrel{(\ref{DIAG_bialgebra_A(L)_univ_prop})}{=}& (f \otimes \theta) \circ \Phi_{\mathfrak{L}} \\
			&=& (f \otimes id_{B}) \circ 	\underbrace{(id_{\mathfrak{L}} \otimes \theta) \circ \Phi_{\mathfrak{L}}}\\
			&\stackrel{(\ref{DIAG_bialgebra_A(L)_univ_prop})}{=}& (f \otimes id_{B}) \circ f \\
			&=& (id_{\mathfrak{L}} \otimes \Delta_{B}) 	\circ f \\
			&\stackrel{(\ref{DIAG_bialgebra_A(L)_univ_prop})}{=}& (id_{\mathfrak{L}} \otimes \Delta_{B}) \circ (id_{\mathfrak{L}} \otimes \theta) \circ \Phi_{\mathfrak{L}} \\
			&=& (id_{\mathfrak{L}} \otimes \Delta_{B} 	\circ \theta) \circ \Phi_{\mathfrak{L}}.
		\end{eqnarray*}
		This shows that $(\theta \otimes \theta) \circ \Phi_{\mathfrak{L}}$ is a valid choice for $\psi$. Again, using the universal property of $\mathcal{A}(\mathfrak{L})$ we get the following diagrams 
		\[\begin{tikzcd}
			{\mathfrak{L}} & {\mathfrak{L} \otimes 	\mathcal{A} (\mathfrak{L})} && {\mathfrak{L}} & {\mathfrak{L} \otimes \mathcal{A} (\mathfrak{L})} \\
			& {\mathfrak{L} \otimes B \otimes B,} &&& 	{\mathfrak{L} \otimes B \otimes B.}
			\arrow["{\Phi_{\mathfrak{L}}}", from=1-1, 	to=1-2]
			\arrow["{\Phi_{\mathfrak{L}}}", from=1-4, 	to=1-5]
			\arrow["{id_{\mathfrak{L}} \otimes (\theta 	\otimes \theta) \circ \Delta}", from=1-2, to=2-2]
			\arrow["{id_{\mathfrak{L}} \otimes 	\Delta_{B} \circ \theta}", from=1-5, to=2-5]
			\arrow["\Omega"', from=1-4, to=2-5]
			\arrow["\Omega"', from=1-1, to=2-2]
		\end{tikzcd}\]
		Since there exists only one algebra morphism from $\mathcal{A}(\mathfrak{L})$ to $B \otimes B$ making the above diagram commutative. Hence, $(\theta \otimes \theta) \circ \Delta ~=~ \Delta_{B} \circ \theta$. Thus giving us the commutative diagram 
		\[\begin{tikzcd}
			{\mathcal{A}(\mathfrak{L})} & B \\
			{\mathcal{A}(\mathfrak{L}) \otimes 	\mathcal{A}(\mathfrak{L})} & {B \otimes B}
			\arrow["\theta", from=1-1, to=1-2]
			\arrow["{\theta \otimes \theta }", 	from=2-1, to=2-2]
			\arrow["\Delta"', from=1-1, to=2-1]
			\arrow["{\Delta_{B}}", from=1-2, to=2-2].
		\end{tikzcd}\]
		
		\noindent
		Similarly, one can show the commutativity of the second diagram in (\ref{DIAG_bialgebra_A(L)_univ_prop_1}), i.e., 
		\[\begin{tikzcd}
			{\mathcal{A}(\mathfrak{L})} & B \\
			& {\mathbb{K}.}
			\arrow["{\theta }", from=1-1, to=1-2]
			\arrow["\varepsilon"', from=1-1, to=2-2]
			\arrow["{\varepsilon_B}", from=1-2, to=2-2]
		\end{tikzcd}\]
		Hence, the map $\theta:\mathcal{A}(\mathfrak{L}) \to B$ is a bialgebra morphism.  
	\end{proof}
	
	\medskip
	Therefore, for any finite-dimensional Lie-Yamaguti algebra $\mathfrak{L}$, we have constructed a commutative bialgebra $\mathcal{A}(\mathfrak{L})$.
We denote the category of commutative bialgebras by $\operatorname{ComBiAlg}_{\mathbb{K}}$ and the category of commutative Hopf algebras by $\operatorname{ComHopf}_{\mathbb{K}}$. From Section \ref{sec-2}, let us recall that a commutative bialgebra generates a free commutative Hopf algebra. Also, recall from \cite{takeuchi71} that assigning the free commutative Hopf algebra to a commutative bialgebra defines a functor $L: \operatorname{ComBiAlg_{\mathbb{K}}} \to \operatorname{ComHopf_{\mathbb{K}}}$ that is left adjoint to the forgetful functor $F:\operatorname{ComHopf_{\mathbb{K}}} \to \operatorname{ComBiAlg_{\mathbb{K}}}$. We denote by $\mu:\mathds{1}_{\operatorname{ComBiAlg_{\mathbb{K}}}} \to FL,$ the unit of the adjunction $L \dashv F$.
	
	\begin{definition}
		Let $\mathfrak{L}$ be a finite-dimensional Lie-Yamaguti algebra over $\mathbb{K}$. The pair $\big(\mathcal{H}(\mathfrak{L}):= L(\mathcal{A}(\mathfrak{L}))~,~\mathcal{X}_{\mathfrak{L}}:= (id_{\mathfrak{L}} \otimes \mu_{\mathcal{A}(\mathfrak{L})}) \circ \Phi_{\mathfrak{L}}\big)$
		is called the {\bf universal coacting Hopf algebra of $\mathfrak{L}$}.
	\end{definition}
	
	The pair $(\mathcal{H}(\mathfrak{L}), \mathcal{X}_{\mathfrak{L}})$ satisfies the following universal property. 	
	\begin{theorem}\label{THM_hopf_univ_prop}
		Let $\mathfrak{L}$ be a finite-dimensional Lie-Yamaguti algebra. Then for any commutative Hopf algebra $H$ and any Lie-Yamaguti algebra morphism $f: \mathfrak{L} \to \mathfrak{L} \otimes H$ that makes $\mathfrak{L}$ into a right $H$-comodule there exists a unique Hopf algebra morphism $g: \mathcal{H}(\mathfrak{L}) \to H$ such that the following diagram is commutative
		\begin{equation}\label{DIAG_hopf_univ_prop}
			\begin{tikzcd}
				{\mathfrak{L}} & {\mathfrak{L} 
				\otimes \mathcal{H}(\mathfrak{L}) 
				} \\
				& {\mathfrak{L} \otimes H}
				\arrow["{\mathcal{X}_{\mathfrak{L}}}", 	from=1-1, 	to=1-2]
				\arrow["{id_{\mathfrak{L}} \otimes g 	}", 	from=1-2, to=2-2]
				\arrow["f"', from=1-1, to=2-2].
			\end{tikzcd}
		\end{equation}
	\end{theorem}
	
	\begin{proof}
		Let $H$ be a commutative Hopf algebra together with a Lie-Yamaguti algebra morphism $f: \mathfrak{L} \to \mathfrak{L} \otimes H$, which makes $\mathfrak{L}$ into right a H-comodule. Using Theorem \ref{THM_bialgebra_A(L)_univ_prop}, we obtain a unique bialgebra morphism $\theta: \mathcal{A}(\mathfrak{L}) \to H$, which makes the following diagram 
		\begin{equation}\label{DIAG_hopf_univ_prop_1}
			\begin{tikzcd}
				{\mathfrak{L}} & {\mathfrak{L} \otimes 		\mathcal{A}(\mathfrak{L})} \\ 
				& {\mathfrak{L} \otimes H}
				\arrow["{\Phi_{\mathfrak{L}}}", 	from=1-1, to=1-2]
				\arrow["{id_{\mathfrak{L}} \otimes 	\theta }", 	from=1-2, to=2-2]
				\arrow["f"', from=1-1, to=2-2].
			\end{tikzcd}
		\end{equation}	
 commutative. The adjunction $L \dashv F$ gives us a unique Hopf algebra morphism $g: L(\mathcal{A}(\mathfrak{L})) \to H$ such that the diagram
		\begin{equation}\label{DIAG_hopf_univ_prop_2}
			\begin{tikzcd}
				{\mathcal{A}(\mathfrak{L})} & 		{L(\mathcal{A}(\mathfrak{L}))} \\
				& H. \\
				\arrow["{\mu_{\mathcal{A}(\mathfrak{L})}}", 	from=1-1, to=1-2]
				\arrow["\theta"', from=1-1, to=2-2]
				\arrow["g", from=1-2, to=2-2]
			\end{tikzcd}
		\end{equation}		
		commutes. It follows that $g: \mathcal{H}(\mathfrak{L})=L(\mathcal{A}(\mathfrak{L})) \to H$ is the unique Hopf algebra morphism such that the diagram (\ref{DIAG_hopf_univ_prop}) commutes. In particular,
		\begin{eqnarray*}
			(id_{\mathfrak{L}} \otimes g) \circ (id_{\mathfrak{L}} \otimes 	\mu_{\mathcal{A}(\mathfrak{L})}) \circ \Phi_{\mathfrak{L}} 
			&=& (id_{\mathfrak{L}} \otimes 	\underbrace{g \circ \mu_{\mathcal{A}(\mathfrak{L})}}) 	\circ \Phi_{\mathfrak{L}} \\
			&\stackrel{(\ref{DIAG_hopf_univ_prop_2})}{=}& (id_{\mathfrak{L}} \otimes \theta) \circ \Phi_{\mathfrak{L}} \\
			&\stackrel{(\ref{DIAG_hopf_univ_prop_1})}{=}& f,
		\end{eqnarray*}
		which leads to the commutativity of the diagram (\ref{DIAG_hopf_univ_prop}). Hence, proving the theorem. 	
	\end{proof}
	
Note that it follows from the above Theorem \ref{THM_hopf_univ_prop} that the pair $(\mathcal{H}(\mathfrak{L}), \mathcal{X}_{\mathfrak{L}})$ is the initial object in the category of all commutative Hopf algebras that coact on the Lie-Yamaguti algebra $\mathfrak{L}$.

\medskip
\section{\large Functors between module categories} \label{sec-4}
	This section gives a representation-theoretic version of Manin-Tambara's constructions for Lie-Yamaguti algebras. We show that these universal constructions are functorial. First, we fix some notations for this section.
	\begin{itemize}
		\item For any given algebra $A$ and a Lie-Yamaguti algebra $\mathfrak{L}$, we denote by $A^{\mathcal{M}}$ the category of left $A$-modules and by $\mathfrak{L}^{\mathcal{M}}$ the category of $\mathfrak{L}$-modules.
		\smallskip
		\item Recall that $\mathcal{A}(\mathfrak{L},\mathfrak{K})$ denotes the universal algebra of $\mathfrak{L}$ and $\mathfrak{K}$, given in Definition \ref{DEF_LYA_universal_algebra}. By an $\mathcal{A}$-module we mean a $\mathcal{A}(\mathfrak{L},\mathfrak{K})$-module.
	
\item Let $\mathfrak{L}$ and $\mathfrak{K}$ be two arbitrary Lie-Yamaguti algebras with $\mathfrak{L}$ being a finite dimensional vector space. Let $\{e_1,e_2,\ldots,e_n\}$ and $\{f_i : i \in I\}$ be two fixed basis of $\mathfrak{L}$ and $\mathfrak{K}$, respectively. Also, let $\{\tau_{ij}^s: i,j,s = 1,\ldots,n\} \cup \{\omega_{ijk}^{s}: i,j,k,s = 1,\ldots,n\}$ be the structure constants of $\mathfrak{L}$, i.e., 
	\begin{equation}\label{EQN_sec4_L_struct_const}
		[e_i,e_j]_{\mathfrak{L}} = \sum_{s=1}^n \tau_{ij}^s e_s, \quad \forall~ i,j=1,\ldots,n \quad \text{and} \quad \{e_i,e_j,e_k\}_{\mathfrak{L}} = \sum_{s=1}^{n} \omega_{ijk}^s e_s, \quad \forall~ i,j,k=1,\ldots,n. 
	\end{equation}
	Let $B_{ij} \subset I$ be a finite subset of $I$ for $i,j \in I$, and $C_{ijk} \subset I$ be a finite subset of $I$ for $i,j,k \in I$ such that 
	\begin{equation}\label{EQN_sec4_K_struct_const}
		[f_i,f_j]_{\mathfrak{K}} = \sum_{u \in B_{ij}} \alpha_{ij}^uf_u \quad \text{and} \quad 
		\{f_i,f_j,f_k\}_{\mathfrak{K}} = \sum_{u \in C_{ijk}} \beta_{ijk}^uf_u.
	\end{equation}	
	\end{itemize}

\noindent In the following result, we show that the tensor product of an $\mathfrak{L}$-module and an $\mathcal{A}$-module is a $\mathfrak{K}$-module.

	\begin{theorem}\label{THM_K_module_UxW}
		Let $(U,\rho, D, \theta)$ be an $\mathfrak{L}$-module and $(W, \cdot)$ be an $\mathcal{A}$-module. Then $U \otimes W$ becomes a $\mathfrak{K}$-module with the following module maps
		\begin{eqnarray}
			\rho_{U \otimes W}(f_p) (u \otimes w) &:=& \sum_{i=1}^{n} \rho(e_i)(u) \otimes x_{ip} \cdot w, \qquad \forall p \in I,\label{EQN_sec4_rep_map1}\\
			D_{U \otimes W}(f_p,f_q)(u \otimes w) &:=& \sum_{i,j=1}^{n} D(e_i,e_j)(u) \otimes (x_{ip}x_{jq}) \cdot w, \qquad \forall~ p,q \in I, \label{EQN_sec4_rep_map2}\\
			\theta_{U \otimes W}(f_p,f_q)(u \otimes w) &:=& \sum_{i,j=1}^{n} \theta(e_i,e_j)(u) \otimes (x_{ip}x_{jq}) \cdot w, \qquad \forall p,q \in I.
			\label{EQN_sec4_rep_map3}
		\end{eqnarray} 
	\end{theorem}
	
	\medskip
	\begin{proof}
		Using the fact $(U,\rho,D,\theta)$ is an $\mathfrak{L}$-module and $\mathcal{A}(\mathfrak{L},\mathfrak{K})$ is a commutative algebra, we prove that $(U \otimes W, \rho_{U \otimes W}, D_{
		U \otimes W}, \theta_{U \otimes W})$ is a $\mathfrak{K}$-module. Let us verify that (\ref{R1}) holds, i.e.,  
		\begin{equation*}
			\underbrace{D_{U \otimes W}(f_p,f_q)}_{(1)} + \underbrace{\theta_{U \otimes W}(f_p,f_q)}_{(2)} - \underbrace{\theta_{U \otimes W}(f_q,f_p)}_{(3)} ~=~ 	
			\underbrace{[\rho_{U \otimes W}(f_p),\rho_{U \otimes W}(f_q)]}_{(4)} - 
			\underbrace{\rho_{U \otimes W}([f_p,f_q])}_{(5)}.
		\end{equation*}
		Evaluating the left-hand side for an arbitrary element $u \otimes w \in U \otimes W$, we get 
		\begin{align*}
			\sum_{i,j=1}^{n} &D(e_i,e_j)(u) \otimes (x_{ip}x_{jq})\cdot w + \sum_{i,j=1}^{n} \theta(e_i,e_j)(u) \otimes (x_{ip}x_{jq})\cdot w
			- \sum_{i,j=1}^n \theta(e_j,e_i)(u) \otimes (x_{jq}x_{ip})\cdot w \\
			&=
			\sum_{i,j=1}^n\Big(D(e_i,e_j)(u) + \theta(e_i,e_j)(u) - \theta(e_j,e_i)(u)\Big) \otimes (x_{ip}x_{jq}) \cdot w \\
			&= \sum_{i,j=1}^{n} \Big([\rho(e_i),\rho(e_j)](u) - \rho([e_i,e_j])(u)\Big) \otimes (x_{ip}x_{jq}) \cdot w\\
			&= 
			\sum_{i,j=1}^{n} [\rho(e_i),\rho(e_j)](u) \otimes (x_{ip}x_{jq}) \cdot w ~- \sum_{i,j=1}^{n} \rho([e_i,e_j])(u) \otimes (x_{ip}x_{jq}) \cdot w\\
			&= 
			\underbrace{\big[\rho_{U \otimes W} (f_p), \rho_{U \otimes W} (f_q)\big](u \otimes w)}_{(4)} 
			~-~ \underbrace{\sum_{i,j=1}^{n} \rho([e_i,e_j])(u) \otimes (x_{ip}x_{jq}) \cdot w}_{(a)}.
		\end{align*}
		To check that (\ref{R1}) holds, it is enough to show $(5) = (a)$. it follows from the following calculation.
		\begin{eqnarray*}
			\rho_{U \otimes W}([f_p,f_q])(u \otimes w) 
			&\stackrel{(\ref{EQN_sec4_K_struct_const})}{=}& \rho_{U \otimes W}\Big(\sum_{r \in B_{pq}}\alpha_{pq}^r f_r\Big)(u \otimes w) \\
			&=& 
			\sum_{r \in B_{pq}} \alpha_{pq}^r ~\rho_{U \otimes W}(f_r)(u \otimes w) \\
			&\stackrel{(\ref{EQN_sec4_rep_map1})}{=}& 	
			\sum_{r \in B_{pq}} \Big(\sum_{s=1}^n \rho(e_s)(u) \otimes x_{sr} \cdot w\Big) \\
			&=& 
			\sum_{s=1}^n \rho(e_s)(u) \otimes \Big(\sum_{r \in B_{pq}} \alpha_{pq}^r ~x_{sr}\Big) \cdot w\\
			&\stackrel{(\ref{EQN_LYA_univ_poly1})}{=}&
			\sum_{s=1}^n \rho(e_s)(u) \otimes \Big(\sum_{i,j=1}^n \tau_{ij}^s x_{ip}x_{jq}\Big) \cdot w \\
			&=& 
			\sum_{i,j=1}^n \Big(\sum_{s=1}^n \tau_{ij}^s \rho(e_s)\Big)(u) \otimes (x_{ip}x_{jq}) \cdot w \\
			&\stackrel{(\ref{EQN_sec4_L_struct_const})}{=}& 
			\sum_{i,j = 1}^n \rho([e_i,e_j]) (u) \otimes (x_{ip}x_{jq}) \cdot w.		
		\end{eqnarray*}
		Similarly, one can verify that (\ref{R2}--\ref{R6}) hold. Thus, making $(U \otimes W, \rho_{U \otimes W}, D_{U \otimes W} , \theta_{U \otimes W})$ a $\mathfrak{K}$-module.   
	\end{proof}

	\smallskip
	\noindent The following proposition shows that the above construction is functorial. 
	
	\begin{proposition}\label{LEM_functor}
		For any object $(U,\rho,D,\theta)$ in $ \mathfrak{L}^{\mathcal{M}}$, the mapping  
		\begin{equation*}
			U \otimes - : \mathcal{A}^{\mathcal{M}} \to \mathfrak{K}^{\mathcal{M}}
		\end{equation*}
		gives a functor from the category of $\mathcal{A}$-modules to the category of $\mathfrak{K}$-modules.
	\end{proposition}
	
	\begin{proof}
	Theorem \ref{THM_K_module_UxW} shows that for any $\mathcal{A}$-module $W$, $U \otimes W$ is a $\mathfrak{K}$-module. For any $\mathcal{A}$-module morphism $g: W \to W'$, let us define $U \otimes g: U \otimes W \to U \otimes W'$ as $U\otimes g:=id_U \otimes g$. Let us observe that
		\begin{align*}
			(id_U \otimes g) \rho_{U \otimes W}(f_i)(u \otimes w) ~&\stackrel{(\ref{EQN_sec4_rep_map1})}{=}~
			\sum_{s=1}^n \rho_U(e_s)(u) \otimes g(x_{si} \cdot w) 
			~=~
			\sum_{s=1}^n  \rho_U(e_s)(u) \otimes x_{si} \cdot g(w) \\
			~&\stackrel{(\ref{EQN_sec4_rep_map1})}{=}~ 
			\rho_{U \otimes W'}(u \otimes g(w)),
		\end{align*}
		for all $i \in I$, $u \in U$, and $w \in W$.
		Also,
		\begin{align*}
			(id_U \otimes g) D_{U \otimes W}(f_i,f_j)(u \otimes w)
			&~\stackrel{(\ref{EQN_sec4_rep_map2})}{=}~
			\sum_{s,p=1}^n D_U(e_s,e_p)(u) \otimes g(x_{si}x_{pj} \cdot w) \\
			&~=~
			\sum_{s,p=1}^n D_U(e_s,e_p)(u) \otimes x_{si}x_{pj} \cdot g(w)
			~\stackrel{(\ref{EQN_sec4_rep_map2})}{=}
			D_{U \otimes W'}(f_i,f_j)(u \otimes w).      
		\end{align*}
		Similarly, it follows that $(id_U \otimes g) \theta_{U \otimes W} (f_i,f_j) = \theta_{U \otimes W'}(f_i,f_j)$. Thus, $U \otimes g: U \otimes W \to U \otimes W'$ is a $\mathfrak{K}$-module morphism.				If $g = id_W$, then $U\otimes id_{W}=id_{U\otimes W}$. Furthermore, if $f:W \to W'$ and $g:W' \to W''$ are two $\mathcal{A}$-module morphisms,  then by definition it follows that $U\otimes (g \circ f) = (U\otimes g) \circ (U\otimes f)$.  Hence, the mapping $U\otimes -$ gives a functor.
	\end{proof}
	
	\medskip
	We now discuss a representation-theoretic version of the universal algebra of $\mathfrak{L}$ and $\mathfrak{K}$ (Definition \ref{DEF_LYA_universal_algebra}).	
	
	\begin{definition}\label{DEF_univ_A_module}
		Let $U$ be an $\mathfrak{L}$-module and $V$ be a $\mathfrak{K}$-module. The {\bf universal $\mathcal{A}$-module of $U$ and $V$} is a pair $\big(\mathcal{U}(U,V), \Gamma_{\mathcal{U}(U,V)}\big)$, where $\mathcal{U}(U,V)$ is an $\mathcal{A} $-module and $\Gamma_{\mathcal{U}(U,V)} : V \to U \otimes \mathcal{U}(U,V)$ is a $\mathfrak{K}$-module morphism such that for any other pair $(W,f)$ consisting of an $\mathcal{A}$-module $W$ and a $\mathfrak{K}$-module morphism $f: V \to U \otimes W$ there exists a unique $\mathcal{A}$-module morphism $g: \mathcal{U}(U,V) \to W$ making the following diagram commutative
		\begin{equation} \label{DIAG_univ_Amodule}
			\begin{tikzcd}
				V & {U \otimes \mathcal{U}(U,V)} \\
				& {U \otimes W}
				\arrow["{\Gamma_{\mathcal{U}(U,V)}}", from=1-1, to=1-2]
				\arrow["{id_{g} \otimes g}", from=1-2, to=2-2]
				\arrow["f"', from=1-1, to=2-2].
			\end{tikzcd}
		\end{equation}
	\end{definition}
	
	\begin{remark}
		Consider a category, whose objects are the pairs $(W,f)$ consisting of an $\mathcal{A}$-module $W$ and a morphism of $\mathfrak{K}$-module $f:V \to U \otimes W$, and a morphism between two objects $(W,f)$ and $(W',f')$ are defined to be $\mathcal{A}$-module map $h: W \to W'$ satisfying $(id_U \otimes h) \circ f = f'$. Then, the universal $\mathcal{A}$-module of $U$ and $V$ (if it exists) is an initial object in this category.  
	\end{remark}

	\begin{remark} \label{REM 4.5}
		Let $U$ be an $\mathfrak{L}$-module, $V$ be a $\mathfrak{K}$-module and $W$ be an $\mathcal{A}$-module. The above definition gives a bijective correspondence between the set of all $\mathfrak{K}$-module maps $f: V \to U \otimes W$ and the set of all $\mathcal{A}$-module maps $g: \mathcal{U}(U, V) \to W$. 
	\end{remark}
	With the above notations, we have the following result on the existence of universal $\mathcal{A}$-module of $U$ and $V$.
		
	\begin{theorem}\label{THM_univ_Amodule_U_and_V}
		If $U$ is a finite dimensional $\mathfrak{L}$-module and $V$ is an arbitrary $\mathfrak{K}$-module then the universal $\mathcal{A}$-module of $U$ and $V$ exists. 
	\end{theorem}
	
	\begin{proof}
		Let $\{u_1,\ldots,u_m\}$ be a $\mathbb{K}$-basis of $\mathfrak{L}$-module. Let $\gamma_{ir}^s,~ \delta_{ijr}^s ,~ \varepsilon_{ijr}^s \in \mathbb{K}$ be the structure constants of $U$ with respect to its $\mathfrak{L}$-module structure $(U,\rho_U,D_U,\theta_U)$, i.e., 
		\begin{eqnarray}
			\rho_{U}(e_i)(u_k) &=& \sum_{s=1}^m \gamma_{ik}^{s} u_s, \qquad \forall i = 1,\ldots, n,~k=1, \ldots, m,  \label{EQN_Lmodule_struct_const1} \\
			D_U(e_i,e_j)(u_k) &=& \sum_{s=1}^m \delta_{ijk}^s u_s, \qquad \forall i,j=1,\ldots,n,~k=1, \ldots, m, \label{EQN_Lmodule_struct_const2} \\
			\theta_U(e_i,e_j)(u_k) &=& \sum_{s=1}^m \varepsilon_{ijk}^s u_s, \qquad \forall i,j=1,\ldots,n, ~k=1, \ldots, m. \label{EQN_Lmodule_struct_const3}
		\end{eqnarray}
 Also, let $\{v_r : r \in J\}$ be a basis of $V$ and let $(V,\rho_V,D_V,\theta_V)$ denote its $\mathfrak{K}$-module structure with structure constants  $\mu_{ir}^s,~ \eta_{ijr}^s ,~ \sigma_{ijr}^s \in \mathbb{K}$. i.e.,
		\begin{eqnarray}
			\rho_V(f_i)(v_r) &=& \sum_{s \in S_{ir}} \mu_{ir}^s v_s, ~~\quad \forall~ i \in I,~~  r \in J, \label{EQN_Kmodule_struct_const1} \\
			D_V(f_i,f_j)(v_r) &=& \sum_{s \in T_{ijr}} \eta_{ijr}^s v_s, \quad \forall~ i,j \in I,~~ r \in J, \label{EQN_Kmodule_struct_const2} \\
			\theta_V(f_i,f_j)(v_r) &=& \sum_{s \in R_{ijr}} \sigma_{ijr}^s v_s, \quad \forall~ i,j \in I,~~ r \in J, \label{EQN_Kmodule_struct_const3}
		\end{eqnarray}
		where $S_{ir}$, $T_{ijr}$, and $R_{ijr}$ are finite subsets of $J$. Now, we construct the universal $\mathcal{A}$-module of $U$ and $V$. 
		
		\medskip
		Let $\mathcal{T}(U,V)$ be the free $\mathcal{A}$-module on the set $\{Y_{sr} : s = 1,\ldots,m, r \in J\}$ and denote by $\mathcal{U}(U,V)$ the quotient of $\mathcal{T}(U,V)$ by its $\mathcal{A}$-submodule generated by the following elements
		\begin{equation} \label{EQN_module_univ_Poly1}
			\sum_{s \in S_{ir}} \mu_{ir}^s Y_{ps} ~-~ \sum_{t=1}^m \sum_{k=1}^m \gamma_{kt}^p ~x_{ki} \bullet Y_{tr}, \quad \forall~ i \in I,~~ r \in J, 
		\end{equation}
		\begin{equation} \label{EQN_module_univ_Poly2}
			\sum_{s \in T_{ijr}} \eta_{ijr}^s Y_{ps} ~-~ \sum_{t=1}^m\sum_{k=1}^m\sum_{l=1}^m \delta_{lkt}^p (x_{li} x_{kj}) \bullet Y_{tr}, \quad \forall~ i,j \in I,~~ r \in J,
		\end{equation}
		\begin{equation} \label{EQN_module_univ_Poly3}
			\sum_{s \in R_{ijr}} \sigma_{ijr}^s Y_{ps} ~-~ \sum_{t=1}^m\sum_{k=1}^m\sum_{l=1}^m \varepsilon_{lkt}^p (x_{li} x_{kj}) \bullet Y_{tr}, \quad \forall~ i,j \in I,~~ r \in J, 
		\end{equation}
		where $p = 1, \ldots, m$ and $\bullet$ denotes the $\mathcal{A}$-module action on $\mathcal{T}(U,V)$. Let $y_{sr} := \widehat{Y}_{sr}$, where $\widehat{Y}_{sr}$ denote the equivalence class of $Y_{sr}$ in $\mathcal{U}(U,V)$. It follows that the following relations hold in $\mathcal{U}(U,V)$ 
		\begin{equation} \label{EQN_module_univ_poly1}
			\sum_{s \in S_{ir}} \mu_{ir}^s y_{ps} ~=~ \sum_{t=1}^m \sum_{k=1}^m \gamma_{kt}^p ~x_{ki} \bullet y_{tr}, \quad \forall~ i \in I,~~ r \in J, 
		\end{equation}
		\begin{equation} \label{EQN_module_univ_poly2}
			\sum_{s \in T_{ijr}} \eta_{ijr}^s y_{ps} ~=~ \sum_{t=1}^m\sum_{k=1}^m\sum_{l=1}^m \delta_{lkt}^p (x_{li} x_{kj}) \bullet y_{tr}, \quad \forall~ i,j \in I,~~ r \in J, 
		\end{equation}
		\begin{equation} \label{EQN_module_univ_poly3}
			\sum_{s \in R_{ijr}} \sigma_{ijr}^s y_{ps} ~=~ \sum_{t=1}^m\sum_{k=1}^m\sum_{l=1}^m \varepsilon_{lkt}^p (x_{li} x_{kj}) \bullet y_{tr}, \quad \forall~ i,j \in I,~~ r \in J,
		\end{equation}
		for all $p = 1, \ldots, m$.
		
		\medskip
		We define a morphism of $\mathfrak{K}$-module, $\Gamma_{\mathcal{U}(U,V)} : V \to U \otimes \mathcal{U}(U,V)$ as follows
		\begin{equation}\label{EQN_module_univ_morphism}
			\Gamma_{\mathcal{U}(U,V)}(v_r) := \sum_{p=1}^m u_p \otimes y_{pr}  \quad \forall~ r \in J. 
		\end{equation}
		Now, we show that $\Gamma_{\mathcal{U}(U,V)}$ is a $\mathfrak{K}$-module morphism by showing that the following diagrams
		\[\begin{tikzcd}
			V && {U \otimes \mathcal{U}(U,V)} &&& V && {U \otimes \mathcal{U}(U,V)} \\
			V && {U \otimes \mathcal{U}(U,V)} &&& V && {U \otimes \mathcal{U}(U,V)}
			\arrow["{\rho_V(f_i)}"', from=1-1, to=2-1]
			\arrow["{\rho_{U \otimes \mathcal{U}(U,V)}(f_i)}", from=1-3, to=2-3]
			\arrow["{\Gamma_{\mathcal{U}(U,V)}}", from=1-1, to=1-3]
			\arrow["{\Gamma_{\mathcal{U}(U,V)}}", from=2-1, to=2-3]
			\arrow["{D_{U \otimes \mathcal{U}(U,V)}(f_i,f_j)}", from=1-8, to=2-8]
			\arrow["{D_V(f_i,f_j)}"', from=1-6, to=2-6]
			\arrow["{\Gamma_{\mathcal{U}(U,V)}}", from=1-6, to=1-8]
			\arrow["{\Gamma_{\mathcal{U}(U,V)}}", from=2-6, to=2-8]
		\end{tikzcd}\]
		\[\begin{tikzcd}
			V && {U \otimes \mathcal{U}(U,V)} \\
			V && {U \otimes \mathcal{U}(U,V)}
			\arrow["{\theta_V(f_i,f_j)}"', from=1-1, to=2-1]
			\arrow["{\theta_{U \otimes \mathcal{U}(U,V)}(f_i,f_j)}", from=1-3, to=2-3]
			\arrow["{\Gamma_{\mathcal{U}(U,V)}}", from=1-1, to=1-3]
			\arrow["{\Gamma_{\mathcal{U}(U,V)}}", from=2-1, to=2-3]
		\end{tikzcd}\]
		commute for all $i,j \in I$. Let us verify the commutativity of the first diagram.  
		\begin{eqnarray*}
			\Gamma_{\mathcal{U}(U,V)}(\rho_{V}(f_i)(v_r)) 
			&\stackrel{(\ref{EQN_Kmodule_struct_const1})}{=}&
			\Gamma_{\mathcal{U}(U,V)}\Big(\sum_{s \in S_{ir}} \mu_{ir}^s v_s\Big) 
			~=~
			\sum_{s \in S_{ir}} \mu_{ir}^s ~\Gamma_{\mathcal{U}(U,V)}(v_s) \\
			&\stackrel{(\ref{EQN_module_univ_morphism})}{=}& 
			\sum_{s \in S_{ir}} \sum_{p=1}^m \mu_{ir}^s u_p \otimes y_{ps} 
			~=~ 
			\sum_{p=1}^m \Big(u_p \otimes \sum_{s \in S_{ir}} \mu_{ir}^s y_{ps}\Big) \\
			&\stackrel{(\ref{EQN_module_univ_poly1})}{=}&
			\sum_{p=1}^m \Big(u_p \otimes \sum_{t=1}^m \sum_{k=1}^m \gamma_{kt}^p x_{ki} \bullet y_{tr}\Big)
			~=~ 
			\sum_{t,k=1}^m \Big(\sum_{p=1}^m \gamma_{kt}^p u_p\Big) \otimes x_{ki} \bullet y_{tr} \\
			&\stackrel{(\ref{EQN_Lmodule_struct_const1})}{=}&
			\sum_{t,k=1}^m \rho_U(e_k)(u_t) \otimes x_{ki} \bullet y_{tr} 
			~=~ 
			\sum_{t=1}^m \Big(\sum_{k=1}^m \rho_{U}(e_k)(u_t) \otimes x_{ki} \bullet y_{tr}\Big)\\
			&\stackrel{(\ref{EQN_sec4_rep_map1})}{=}&
			\sum_{t=1}^m \Big(\rho_{U \otimes \mathcal{U}(U,V)}(f_i) (u_t \otimes y_{tr})\Big) 
			~=~ 
			\rho_{U \otimes \mathcal{U}(U,V)}(f_i) \Big(\sum_{t=1}^m u_t \otimes y_{tr}\Big)\\
			&\stackrel{(\ref{EQN_module_univ_morphism})}{=}&
			\rho_{U \otimes \mathcal{U}(U,V)}(f_i)\big(\Gamma_{\mathcal{U}(U,V)}(v_r)\big).
		\end{eqnarray*}
		To show that the second diagram commutes, we consider the following expression.
		\begin{eqnarray*}
			\Gamma_{\mathcal{U}(U,V)}(D_V(f_i,f_j)(v_r)) 
			&\stackrel{(\ref{EQN_Kmodule_struct_const2})}{=}&
			\Gamma_{\mathcal{U}(U,V)}\Big(\sum_{s \in T_{ijr}} \eta_{ijr}^s v_s\Big) 
			~=~ 
			\sum_{s \in T_{ijr}} \eta_{ijr}^s ~\Gamma_{\mathcal{U}(U,V)}(v_s) \\
			&\stackrel{(\ref{EQN_module_univ_morphism})}{=}&
			\sum_{s \in T_{ijr}} \eta_{ijr}^s \Big(\sum_{p=1}^{m} u_p \otimes y_{ps}\Big)
			~=~ 
			\sum_{p=1}^m u_p \otimes \sum_{s \in T_{ijr}} \eta_{ijr}^s y_{ps} \\
			&\stackrel{(\ref{EQN_module_univ_poly2})}{=}&
			\sum_{p=1}^m u_p \otimes \sum_{t=1}^m \sum_{k=1}^m \sum_{l=1}^m \delta_{lkt}^p (x_{li} x_{kj}) \bullet y_{tr} \\
			&=&
			\sum_{t,k,l=1}^m \Big(\sum_{p=1}^m \delta_{lkt}^p u_p\Big) \otimes (x_{li} x_{kj}) \bullet y_{tr} \\
			&\stackrel{(\ref{EQN_Lmodule_struct_const2}),(\ref{EQN_sec4_rep_map2})}{=}& 
			\sum_{t=1}^m D_{U \otimes \mathcal{U}(U,V)}(f_i,f_j) (u_t \otimes y_{tr}) 
			~=~
			D_{U \otimes \mathcal{U}(U,V)}(f_i,f_j) \Big(\sum_{t=1}^m u_t \otimes y_{tr}\Big)\\
			&\stackrel{(\ref{EQN_module_univ_morphism}) }{=}&
			D_{U \otimes \mathcal{U}(U,V)}(f_i,f_j)\big(\Gamma_{\mathcal{U}(U,V)}(v_r)\big).
		\end{eqnarray*}
		Similarly, one can show that the third diagram is also commutative. Hence, the map $\Gamma_{\mathcal{U}(U,V)} : V \to U \otimes \mathcal{U}(U,V)$ is a $\mathfrak{K}$-module morphism. 
		
		\medskip
		It is left to show that the pair $\big(\mathcal{U}(U,V), \Gamma_{\mathcal{U}(U,V)}\big)$ satisfies the universal property \eqref{DIAG_univ_Amodule}. Consider a pair $(W,f)$ consisting of an $\mathcal{A}$-module $W$ and a morphism of $\mathfrak{K}$-modules $f: V \to U \otimes W$. Let $\{w_{sr}: s = 1, \ldots, m, r \in J\}$ be a family of elements of $W$ such that 
		\begin{equation}\label{EQN_W_morphism}
			f(v_r) = \sum_{s=1}^m u_s \otimes w_{sr}, \quad \forall~ r \in J.  
		\end{equation}
		Since $f: V \to U \otimes W$ is a $\mathfrak{K}$-module morphism, we have the following identities
		\begin{equation}\label{EQN_Wmodule_univ_poly1}
			\sum_{s \in S_{ir}} \mu_{ir}^s w_{ps} - \sum_{t=1}^m \sum_{k=1}^m \gamma_{kt}^p x_{ki} \cdot w_{tr},
		\end{equation}
		\begin{equation}\label{EQN_Wmodule_univ_poly2}
			\sum_{s \in T_{ijr}} \eta_{ijr}^s w_{ps} - \sum_{t=1}^m\sum_{k=1}^m \sum_{l=1}^m \delta_{lkt}^p (x_{li}x_{kj}) \cdot w_{tr}, 	
		\end{equation}
		\begin{equation}\label{EQN_Wmodule_univ_poly3}
			\sum_{s \in R_{ijr}} \sigma_{ijr}^s w_{ps} - \sum_{t=1}^m\sum_{k=1}^m \sum_{l=1}^m \varepsilon_{lkt}^p (x_{li}x_{kj}) \cdot w_{tr}, 
		\end{equation}
		for all $p=1,\ldots,m$, $i,j \in I$ and $r \in J$. Here, $\cdot$ denotes the action of $\mathcal{A}$-module action on $W$. Now, the universal property of the free module yields a unique $\mathcal{A}$-module morphism 
		\[\overline{g}: \mathcal{T}(U,V) \to W\] 
		such that $\overline{g}(Y_{sr}) = w_{sr}$ for all $s=1,\ldots, m$ and $r \in J$. Furthermore, $\operatorname{Ker}(\overline{g})$ contains the $\mathcal{A}$-submodule of $\mathcal{T}(U,V)$ generated by the elements given in (\ref{EQN_module_univ_Poly1}-- \ref{EQN_module_univ_Poly3}). Thus, there exists a unique $\mathcal{A}$-module morphism $g: \mathcal{U}(U,V) \to W$ such that $g(y_{sr}) = w_{sr}$ for all $s=1,\ldots,m$ and $r \in J$, i.e., 
		\begin{eqnarray*}
			(id_U \otimes g) \circ \Gamma_{\mathcal{U}(U,V)} (v_r) 
			&=& 
			id_U \otimes g \Big(\sum_{s=1}^m u_s \otimes y_{sr}\Big) 
			~=~
			\sum_{s=1}^m u_s \otimes w_{sr} = f(v_r), \quad \forall r \in J.
		\end{eqnarray*} 
		Therefore, the diagram (\ref{DIAG_univ_Amodule}) is commutative. Verifying that $g$ is the unique $\mathcal{A}$-module morphism with this property is straightforward.
\end{proof}
	
	\begin{theorem}\label{THM_univ_module_functor}
		If $U$ is a finite-dimensional $\mathfrak{L}$-module, then one can construct a functor $\mathcal{U}: \mathfrak{K}^{\mathcal{M}} \to \mathcal{A}^{\mathcal{M}}$ defined by
		\begin{equation*}
			\mathcal{U}(V) := \mathcal{U}(U,V), \quad \mathcal{U}(f) = \bar{f}, 
		\end{equation*} 
		where for any $\mathfrak{K}$-module morphism $f:V \to V'$ the map $\bar{f}: \mathcal{U}(U,V) \to \mathcal{U}(U,V')$ is the unique $\mathcal{A}$-module morphism that makes the following diagram 
		\begin{equation}\label{DIAG_module_functor1}
			\begin{tikzcd}
				V && {U \otimes \mathcal{U}(U,V)} \\
				&& {U \otimes \mathcal{U}(U,V')}
				\arrow["{\Gamma_{\mathcal{U}(U,V)}}", from=1-1, to=1-3]
				\arrow["{\Gamma_{\mathcal{U}(U,V')} \circ f}"', from=1-1, to=2-3]
				\arrow["{id_{U} \otimes \bar{f}}", from=1-3, to=2-3]
			\end{tikzcd} 
		\end{equation} 
		commutative. Furthermore, the functors 
		$U \otimes -:\mathcal{A}^{\mathcal{M}} \to \mathfrak{K}^{\mathcal{M}}$ and $\mathcal{U}: \mathfrak{K}^{\mathcal{M}} \to \mathcal{A}^{\mathcal{M}}$ are adjoint to each other. 
	\end{theorem}
	
	\begin{proof}
		If $f = id_V$, then $id_{\mathcal{U}(U,V)}$ is the unique $\mathcal{A}$-module morphism for which the diagram (\ref{DIAG_module_functor1}) commutes and therefore $\mathcal{U}(id_V) = id_{\mathcal{U}(U,V)}$. If $f:V \to V'$ and $g:V' \to V''$ are two morphisms in $\mathfrak{K}^{\mathcal{M}}$, then $\bar{g} \circ \bar{f} : \mathcal{U}(U,V) \to \mathcal{U}(U,V'')$ is the unique $\mathcal{A}$-module morphism such that the diagram
		\begin{equation*}
			\begin{tikzcd}
				V && {U \otimes \mathcal{U}(U,V)} \\
				&& {U \otimes \mathcal{U}(U,V'')}
				\arrow["{\Gamma_{\mathcal{U}(U,V)}}", from=1-1, to=1-3]
				\arrow["{id_{U} \otimes (\bar{g} \circ \bar{f})}", from=1-3, to=2-3]
				\arrow["{\Gamma_{\mathcal{U}(U,V'')} \circ g \circ f}"', from=1-1, to=2-3]
			\end{tikzcd}
		\end{equation*}
		commutes, implying $\mathcal{U}_U(g \circ f) = \mathcal{U}_U(g) \circ \mathcal{U}_U(f)$.  
		
From the Remark \ref{REM 4.5}, it follows that the following diagram 
		\begin{equation*}
			\begin{tikzcd}
				{\operatorname{Hom}_{\mathcal{A}^{\mathcal{M}}}(\mathcal{U}(V),W)} & {\operatorname{Hom}_{\mathcal{A}^{\mathcal{M}}}(\mathcal{U}(V'),W)} & {\operatorname{Hom}_{\mathcal{A}^{\mathcal{M}}}(\mathcal{U}(V'),W')} \\
				{\operatorname{Hom}_{\mathfrak{K}^{\mathcal{M}}}(V,U \otimes W)} & {\operatorname{Hom}_{\mathfrak{K}^{\mathcal{M}}}(V',U \otimes W)} & {\operatorname{Hom}_{\mathfrak{K}^{\mathcal{M}}}(V',U \otimes W')}
				\arrow[from=1-1, to=1-2]
				\arrow[from=1-2, to=1-3]
				\arrow[from=1-3, to=2-3]
				\arrow[from=1-2, to=2-2]
				\arrow[from=1-1, to=2-1]
				\arrow[from=2-1, to=2-2]
				\arrow[from=2-2, to=2-3]
			\end{tikzcd}
		\end{equation*}
		commutative, giving naturality in both variables. Therefore, $U \otimes -$ and $\mathcal{U}$ are adjoint functors.  
	\end{proof}

\medskip
\section{\large Applications} \label{sec-5}
	In this section, we give some applications of our results. In particular, Theorem \ref{THM_application_1} describes the automorphism group of a finite-dimensional Lie-Yamaguti algebra. For an abelian group $G$, Theorem \ref{THM_application_2} classifies all $G$-gradings on a finite-dimensional Lie-Yamaguti algebra. 

\medskip	
\subsection*{A characterization of the automorphism group of a finite-dimensional Lie-Yamaguti algebra}
	We recall some basic facts about bialgebras and Hopf algebras, which we use in this section.
	\begin{itemize}
		\item For any bialgebra $B$, we denote the set of all group-like elements by
		\begin{equation*}
			G(B) := \{g \in B : \Delta(g) = g \otimes g ~\text{and}~ \varepsilon(g) = 1\}.
		\end{equation*}
		$G(B)$ is a monoid with the multiplication of $B$.
		\item Let $B^\circ$ be the finite dual bialgebra of $B$, i.e., 
		\[B^\circ = \{f \in B^{*} = \operatorname{Hom}(B,\mathbb{K}): f(I) = 0, \text{for some ideal}~ I ~\mbox{of}~ B ~\text{with}~ dim_{\mathbb{K}}(B/I) < \infty\}.\]  
		Here, $\operatorname{Hom}(B, \mathbb{K})$ denotes the vector space of all linear maps from $B$ to $\mathbb{K}$. One can easily show that 
		\[G(B^\circ) = \operatorname{Hom}_{\operatorname{Alg}_{\mathbb{K}}}(B,\mathbb{K}),\]
		where $\operatorname{Hom}_{\operatorname{Alg}_{\mathbb{K}}}(B,\mathbb{K})$
		is the set of all algebra morphism $B \to \mathbb{K}$. 
	\end{itemize}  
	\begin{theorem} \label{THM_application_1}
		Let $\mathfrak{L}$ be a finite-dimensional Lie-Yamaguti algebra with basis $\{e_1,e_2,\ldots,e_n\}$ and consider $G(\mathcal{A}(\mathfrak{L})^\circ)^{inv}$ to be the group of all invertible group-like elements of the finite dual $\mathcal{A}(\mathfrak{L})^\circ$. Then the map $\zeta$ defined by 
		\begin{equation}\label{EQN_app_automorphism}
			 {\zeta}: G(\mathcal{A}(\mathfrak{L})^\circ)^{inv} \to \operatorname{Aut_{LYA}}(\mathfrak{L}), \quad  {\zeta}(\theta)(e_i) := \sum_{s=1}^{n} \theta(x_{si})e_s,
		\end{equation} 
		for any $\theta \in G(\mathcal{A}(\mathfrak{L})^\circ)^{inv}$ and $i=1,2,\ldots,n$, is an isomorphism of groups.
	\end{theorem}
	
	\begin{proof}
		Using Corollary \ref{COR_LYA_hom_set_bijection} with $\mathfrak{K} = \mathfrak{L}$ and $A = \mathbb{K}$, we have a bijection $\Psi$ defined by 
		\[\Psi: \operatorname{Hom}_{\operatorname{Alg}_{\mathbb{K}}}(\mathcal{A}(\mathfrak{L}), \mathbb{K}) \to \operatorname{Hom}_{\operatorname{LYA}_{\mathbb{K}}}(\mathfrak{L},\mathfrak{L}), \quad \Psi(\theta) = (id_{\mathfrak{L}} \otimes \theta) \circ \Phi_{\mathfrak{L}}.\]
Note that $\operatorname{Hom}_{\operatorname{Alg}_{\mathbb{K}}}(\mathcal{A}(\mathfrak{L}), \mathbb{K}) = G(\mathcal{A}(\mathfrak{L})^\circ)$. Now, we show that $\Psi$ is an isomorphism of monoids. 
		Recall that the monoid structure on $\operatorname{Hom}_{\operatorname{LYA}_{\mathbb{K}}}(\mathfrak{L},\mathfrak{L})$ is given by the usual function composition, and on $G(\mathcal{A}(\mathfrak{L})^\circ)$ is given by the convolution product
		\begin{equation}\label{EQN_app_auto_convolution_product}
			\theta_1 \star \theta_2 (x_{sj}) = \sum_{t=1}^n\theta_1(x_{st})\theta_2(x_{tj}),
		\end{equation} 
		for all $\theta_1, \theta_2 \in G(\mathcal{A}(\mathfrak{L})^\circ)$ and $j,s= 1,2,\ldots,n$. For any $\theta_1, \theta_2 \in G(\mathcal{A}(\mathfrak{L})^\circ)$ and $j = 1,2,\ldots,n$, we have
		\begin{eqnarray*}
			\Psi(\theta_1) \circ \Psi(\theta_2)(e_j) 
			&=& \Psi(\theta_1) \left(\sum_{t=1}^n \theta_2(x_{tj})e_t\right) ~=~ \sum_{s,t=1}^{n} \theta_1(x_{st}) \theta_2(x_{tj}) e_s \\
			&=& \sum_{s=1}^{n} \left(\sum_{t=1}^{n} \theta_1(x_{st}) \theta_2(x_{tj})\right) e_s ~=~ \sum_{s=1}^{n}(\theta_1 \star \theta_2) (x_{sj}) e_s 
			~=~ \Psi(\theta_1 \star \theta_2)(e_j).
		\end{eqnarray*}
		Thus, giving us $\Psi(\theta_1 \star \theta_2) = \Psi(\theta_1) \circ \Psi(\theta_2)$. We also need to show that $\Psi$ preserves the unit of $G(\mathcal{A}(\mathfrak{L})^\circ)$. Since the unit $1$ of the monoid $G(\mathcal{A}(\mathfrak{L})^\circ)$ is the counit $\varepsilon_{\mathcal{A}(\mathfrak{L})}$ of the bialgebra $\mathcal{A}(\mathfrak{L})$, we have
		\[\Psi(1)(e_i)~=~ \Psi(\varepsilon_{\mathcal{A}(\mathfrak{L})})(e_i) ~=~ \sum_{s=1}^{n} \varepsilon_{\mathcal{A}(\mathfrak{L})}(x_{si})e_s ~=~ \sum_{s=1}^{n} \delta_{si}(e_s) ~=~ e_i ~=~ id_{\mathfrak{L}}(e_i).\] 
		Thus, $\Psi$ is an isomorphism of monoid. Hence, $\Psi$ induces an isomorphism $\zeta$ between invertible elements of the monoids given by 
		\begin{equation*}
			 {\zeta}: G(\mathcal{A}(\mathfrak{L})^\circ)^{inv} \to \operatorname{Aut_{LYA}}(\mathfrak{L}), \quad  {\zeta}(\theta)(e_i) := \sum_{s=1}^{n} \theta(x_{si})e_s,\quad \mbox{for any }\theta \in G(\mathcal{A}(\mathfrak{L})^\circ)^{inv}\mbox{ and } i=1,2,\ldots, n.
		\end{equation*} 
	\end{proof}

\medskip	
\subsection*{Classification of the abelian group gradings on a finite-dimensional Lie-Yamaguti algebra}
	Let $G$ be an abelian group and $\mathfrak{L}$ be a Lie-Yamaguti algebra. A $G$-grading on $\mathcal{L}$ is a vector space decomposition $\mathfrak{L} = \oplus_{g \in G} \mathfrak{L}_{g}$ such that 
	\[[\mathfrak{L}_a, \mathfrak{L}_b] \subseteq \mathfrak{L}_{ab},~~ \forall~ a,b \in G  \quad \text{and} \quad \{\mathfrak{L}_a, \mathfrak{L}_b, \mathfrak{L}_c\} \subseteq \mathfrak{L}_{abc},~~ \forall~ a,b,c \in G. \]
	
	\medskip
	\noindent
	Let $\mathbb{K}[G]$ denote the group algebra of $G$. In the following lemma, we discuss the relationship between $G$-gradings on $\mathfrak{L}$ and bialgebra maps $\theta: \mathcal{A}(\mathfrak{L}) \to \mathbb{K}[G]$.
	
	\begin{remark}
For a bialgebra map $\theta: \mathcal{A}(\mathfrak{L}) \to \mathbb{K}[G]$, we can associate a $G$-grading on $\mathfrak{L}$ as follows:  
		\[\mathfrak{L} = \oplus_{g \in G} \mathfrak{L}_g^{(\theta)},\quad \mathfrak{L}_{g}^{(\theta)} ~:=~ \{x \in \mathfrak{L}: (id_{\mathfrak{L}} \otimes \theta) \circ \Phi_{\mathfrak{L}}(x) ~=~ x \otimes g\}, \quad \forall~ g \in G.\]
	\end{remark} 
	
	\begin{lemma}\label{LEM_app_Ggrading} 
		Let $G$ be an abelian group and $\mathfrak{L}$ be a finite-dimensional Lie-Yamaguti algebra. Then there exists a bijection between the set of all $G$-gradings on $\mathfrak{L}$, and the set of all bialgebra maps $\theta:\mathcal{A}(\mathfrak{L}) \to \mathbb{K}[G]$. 
	\end{lemma}
	
	\begin{proof}
		By applying Theorem \ref{THM_bialgebra_A(L)_univ_prop} for the commutative bialgebra $B = \mathbb{K}[G]$, we get a bijection between the set of all bialgebra morphisms $\mathcal{A}(\mathfrak{L}) \to \mathbb{K}[G]$ and the set of all Lie-Yamaguti algebra morphism $f: \mathfrak{L} \to \mathfrak{L} \otimes \mathbb{K}[G]$ that make $\mathfrak{L}$ a right $\mathbb{K}[G]$-comodule. 
		There exists a bijection between the set of all right $\mathbb{K}[G]$-comodule structures $f: \mathfrak{L} \to \mathfrak{L} \otimes \mathbb{K}[G]$ on the vector space $\mathfrak{L}$ and the set of all vector space decompositions $\mathfrak{L} 
		= \oplus_{a \in G} \mathfrak{L}_{a}$. In particular, the bijection is given by
		$$x_{a} \in \mathfrak{L}_{a} \quad \mbox{if and only if} \quad f(x_{a}) = x_{a} \otimes a, \forall~ a \in G.$$
		Therefore, it only remains to show that under this bijection, a right coaction $f: \mathfrak{L} \to \mathfrak{L} \otimes \mathbb{K}[G]$ is a Lie-Yamaguti algebra morphism if and only if $[\mathfrak{L}_a, \mathfrak{L}_{b}] \subseteq \mathfrak{L}_{a b}$ and $\{\mathfrak{L}_a, \mathfrak{L}_{b}, \mathfrak{L}_{c}\} \subseteq \mathfrak{L}_{a b c}$ for all $a, b, c \in G$. We show this below. 
		
		\medskip\noindent
		Let $a, b \in G$ and $x_a \in \mathfrak{L}_{a},~ x_{b}\in \mathfrak{L}_{b}$: then $[f(x_{a}),f(x_{b})] = [x_a \otimes a, x_{b} \otimes b] = [x_{a},x_{b}] \otimes a b$. Thus, we obtain $f([x_a, x_b]) = [f(x_a), f(x_{b})]$ if and only if $[x_a,x_b] \in \mathfrak{L}_{ab}$. Similarly, we have  $f\{x_a,x_b, x_c\} = \{f(x_a), f(x_b), f(x_c)\}$ if and only if $\{x_a, x_b, x_c\} \in \mathfrak{L}_{abc}$. Hence, $f: \mathfrak{L} \to \mathfrak{L} \otimes \mathbb{K}[G]$ is a Lie-Yamaguti algebra morphism if and only if $[\mathfrak{L}_a, \mathfrak{L}_{b}] \subseteq \mathfrak{L}_{ab}$ and $\{\mathfrak{L}_{a}, \mathfrak{L}_{b}, \mathfrak{L}_{c}\} \subseteq \mathfrak{L}_{abc}$ for all $a, b, c \in G$. Hence, completing the proof.   
	\end{proof}
	
	\medskip
	We now classify all $G$-gradings on a given Lie-Yamaguti algebra $\mathfrak{L}$, where $G$ is an abelian group. Two $G$-gradings $\mathfrak{L} = \oplus_{a \in G} \mathfrak{L}_{a}$ and $\mathfrak{L} = \oplus_{a \in G} \mathfrak{L'}_{a}$ on $\mathfrak{L}$ are said to be {\bf isomorphic} if there exists an automorphism  $\omega \in \operatorname{Aut}_{\operatorname{LYA}}(\mathfrak{L})$ of $\mathfrak{L}$ such that $\omega(\mathfrak{L}_{a}) \subseteq \mathfrak{L}'_a$ for all $a \in G$. Since $\omega$ is bijective and $\mathfrak{L}$ is $G$-graded, we can show that the last condition is equivalent to $\omega(\mathfrak{L}_{a}) = \mathfrak{L}'_{a}$ for all $a \in G$. 
	
	\begin{definition}
		Let $G$ be an abelian group and $\mathfrak{L}$ be a finite-dimensional Lie-Yamaguti algebra. Two bialgebra morphisms $\theta_1,\theta_2 : \mathcal{A}(\mathfrak{L}) \to \mathbb{K}[G]$ are called {\bf conjugate} if there exists $g \in G(\mathcal{A}(\mathfrak{L})^\circ)^{inv}$ an invertible group like element of the finite dual $\mathcal{A}(\mathfrak{L})^{\circ}$ such that $\theta_2 = g\star\theta_1\star g^{-1}$, in the convolution algebra $\operatorname{Hom}(\mathcal{A}(\mathfrak{L}), \mathbb{K}[G])$. 
	\end{definition}
	
	Let $\operatorname{Hom}_{\operatorname{BiAlg}_{\mathbb{K}}}(\mathcal{A}(L), \mathbb{K}[G])/\thickapprox$ denote the quotient of the set of all bialgebra morphisms $\mathcal{A}(\mathfrak{L}) \to \mathbb{K}[G]$ by the above conjugacy relation and let $\widehat{\theta}$ denote the equivalence class of $\theta \in \operatorname{Hom}_{\operatorname{BiAlg}_{\mathbb{K}}}(\mathcal{A}(\mathfrak{L}), \mathbb{K}[G])$. Then, using a similar argument to the case of Lie and Leibniz algebras \cite[Theorem 3.5]{agore20}, one can conclude the following result.
		
	\begin{theorem} \label{THM_application_2}
		Let $G$ be an abelian group, and $\mathfrak{L}$ be a finite-dimensional Lie-Yamaguti algebra. Let $\mathbf{G-gradings(\mathfrak{L})}$ be the set of isomorphism classes of all $G$-gradings on $\mathfrak{L}$. Then the map
		\begin{equation}\label{EQN_grading_classification_map}
			\operatorname{Hom}_{\operatorname{BiAlg}_{\mathbb{K}}}(\mathcal{A}(\mathfrak{L}), \mathbb{K}[G])/ \thickapprox ~
			\longrightarrow ~~ \mathbf{G-gradings(\mathfrak{L})}, \quad \widehat{\theta} \mapsto \mathfrak{L}^{(\theta)} := \oplus_{g \in G} \mathfrak{L}_{g}^{(\theta)}
		\end{equation}
		where $\mathfrak{L}_{g}^{\theta} = \{x \in \mathfrak{L}: (id_\mathfrak{L} \otimes \theta) \circ \Phi_{\mathfrak{L}} = x \otimes g\}$, for all $g \in G$, is bijective. \qed
	\end{theorem}


\begin{thebibliography}{B}
		\bibitem{abdelwahab22} H. Abdelwahab, E. Barreiro, A. J. Calderón, and A. F. Ouaridi: \textit{The classification of nilpotent Lie-Yamaguti algebras}, Linear Algebra and its Applications 654,  339--378 (2022).
		
		\bibitem{agore20} A. L. Agore and G. Militaru: \textit{A new invariant for finite dimensional Leibniz/Lie algebras}, Journal of Algebra 562, 390--409 (2020).
		
		
		\bibitem{agore23} A. L. Agore: \textit{Functors between representation categories. Universal modules}, Linear Algebra and its Applications 688, 104--119 (2024).
		
		
		
		\bibitem{benito} P. Benito, C. Draper, and A. Elduque: \textit{Lie Yamaguti algebra related to $\mathfrak{g}_2$}, J. Pure Appl. Algebra 202, 22--54 (2005).
		
		\bibitem{benito-irreducible} P. Benito, A. Elduque, and F. Martín-Herce: \textit{Irreducible Lie–Yamaguti algebras}, J. Pure Appl. Algebra 213, 795--808 (2009). 
		
		
		\bibitem{our-paper} S. Goswami, S. K. Mishra, and G. Mukherjee: \textit{Automorphisms of extensions of Lie-Yamaguti algebras and inducibility problem}, Journal of Algebra 641, 268--306 (2024).
		
		\bibitem{ripan-asia} S. Guo, B. Mondal, and R. Saha: \textit{On equivariant Lie–Yamaguti algebras and related structures}, Asian-European Journal of Mathematics 16, Issue 2 (2023).     
		
		\bibitem{Jacobson}	
		N. Jacobson:  \textit{Lie and Jordan triple systems}, American Journal of Mathematics 71(1), 149--170 (1949).
		
		\bibitem{kikkawa72} M. Kikkawa: \textit{On Locally Reductive Spaces and Tangent Algebras}, Mem. Fac. Lit. \& Sci. Shimane Univ Nat. Sci. 5, 1 -- 13 (1972).
		
		\bibitem{kikk-geo} M. Kikkawa: \textit{Geometry of homogeneous Lie loops}, Hiroshima Math. J. 5(2), 141--179 (1975).
		
		\bibitem{kikk-solv} M. Kikkawa: \textit{Remarks on solvability of Lie triple algebras}, Mem. Fac. Sci. Shimane Univ. 13,  17--22 (1979).
		
		\bibitem{kikk-kill} M. Kikkawa: \textit{On Killing–Ricci forms of Lie triple algebras}, Pacific J. Math. 96(1),  153--161 (1981).
		
		\bibitem{kinyon-weinstein} M. K. Kinyon and A. Weinstein: \textit{Leibniz algebras, Courant algebroids, and multiplications on reductive
		homogeneous spaces}, Amer. J. Math. 123(3), 525--550 (2001).
		
		\bibitem{manin88} Y. I. Manin: \textit{Quantum Groups and Noncommutative Geometry}, Universite de Montreal, Centre de
		Recherches Mathematiques, Montreal, QC (1988).
		
		\bibitem{mil22} G. Militaru: 
		\textit{The automorphisms group and the classification of gradings of finite dimensional associative algebras}, Results in Mathematics 77, 1--13 (2022).
		
		\bibitem{ripan-JOA} B. Mondal and R. Saha: \textit{Deformation cohomology of morphisms of Lie-Yamaguti algebras}, Journal of Algebra (In press) (2024).  
		
		\bibitem{nomizu54} K. Nomizu: \textit{Invariant affine connections on Homogeneous spaces}, American Journal of Mathematics 76(1), 33--65 (1954).
		
		\bibitem{Bergh} T. Raedschelders and M. V. den Bergh: \textit{The Manin Hopf algebra of a Koszul Artin-Schelter regular algebra is quasi-hereditary}, Adv. Math. 305, 601--660 (2017).
		
		
		
		
		
		\bibitem{sagle-sim-anti} A. Sagle: \textit{A note on simple anti-commutative algebras obtained from reductive homogeneous spaces}, Nagoya Math. J. 31, 105--124 (1968).
		
		
		
		\bibitem{sweedler69} M. E. Sweedler: \textit{Hopf Algebras}, Mathematics Lecture Note Series (1969). 
		
		\bibitem{takeuchi71} M. Takeuchi: \textit{Free Hopf algebras generated by coalgebras}, J. Math. Soc. Jpn. 23, 561--582 (1971).
		
		\bibitem{tambara90} D. Tambara: \textit{The coendomorphism bialgebra of an algebra}, J. Fac. Sci., Univ. Tokyo, Sect. 1A,
		Math. 37, 425--456 (1990).
		
		\bibitem{yamaguti58} K. Yamaguti: \textit{On Lie triple systems and its generalization}, Journal of Science of the Hiroshima University, Ser. A  21(3), (1958).
		
		\bibitem{yama-malcev} K. Yamaguti: \textit{On the theory of Malcev algebra}, Kumamoto J. Sci. A 6,
		9--45 (1963).
		
		
		\bibitem{zhang} T. Zhang and J. Li: \textit{Deformations and extensions of Lie-Yamaguti algebra}, Linear and Multilinear Algebra 63(11), 2212--2231 (2015).
	\end{thebibliography}
\end{document}